\pdfoutput=1
\documentclass[11pt]{article}
\usepackage{amsmath,amsfonts,amstext,latexsym,amssymb,amsbsy,amsopn,amsthm,eucal,slashed,enumerate}
\usepackage{mathrsfs}
\usepackage{cases}
\usepackage{array}
\usepackage{color}
\usepackage{txfonts}
\usepackage{graphicx}
\usepackage{comment}
\allowdisplaybreaks[4]
\usepackage[colorlinks,
linkcolor=blue,
anchorcolor=blue,
citecolor=blue]{hyperref}

\newtheorem{theorem}{Theorem}[section]
\newtheorem{proposition}[theorem]{Proposition}
\newtheorem{lemma}[theorem]{Lemma}
\newtheorem{corollary}[theorem]{Corollary}
\newtheorem{definition}[theorem]{Definition}
\numberwithin{equation}{section}
\theoremstyle{remark}
\newtheorem{remark}[theorem]{Remark}
\newtheorem{theoremalph}{Theorem}

\newcommand{\Ric}{\operatorname{Ric}}
\newcommand{\supp}{\operatorname{supp}}

\newcommand{\vol}{\operatorname{Vol}}
\newcommand{\dd}{\,\mathrm d}
\newcommand{\Q}{\mathcal Q}
\newcommand{\cL}{\mathcal L}

\newcommand{\bta}{\begin{theoremalph}}
	\newcommand{\eeta}{\end{theoremalph}}

\begin{document}

	\title
	{Differential Harnack Estimates and Frequency Monotonicity for Filtration Equations on Riemannian Manifolds}
	
	\author{Jian-Hua Hao \quad Yu-Zhao Wang\footnote{Y.-Z. Wang was supported by the Fundamental Research Program of Shanxi Province (Grant No. 202303021211001).} }
	\date{}
	\maketitle

	\begin{abstract}
		In this work we derive local differential Harnack estimates for positive solutions of filtration equations on complete Riemannian manifolds with a lower Ricci curvature bound by using Nash--Moser iteration. Inspired by earlier work on parabolic frequency functions, we establish some new frequency monotonicity formulas for these equations. The results cover the porous-medium and
		fast-diffusion equations, as well as several non-power diffusion laws.
	\end{abstract}
	
	
	
	\section{Introduction and main results}\label{sec:intro}
	In this paper, we study the filtration equation
	\begin{equation}\label{eq:general-equation-intro}
		u_t=\Delta F(u),\qquad u>0.
	\end{equation}
	The existence and regularity theory for filtration equations is well
	developed \cite{Vazquez2007}.  We assume that
	$F\in C^2(\mathcal I)$ and $F'>0$ on an open interval
	$\mathcal I\subset(0,\infty)$ containing the range of the solution. When $F(u)=u$, equation
	\eqref{eq:general-equation-intro} is the classical heat equation, while
	$F(u)=u^m$ gives the porous medium equation for $m>1$ and the fast diffusion
	equation for $0<m<1$. Early studies of concentration- and
	density-dependent diffusion, including power-law models can be found in
	\cite{Pattle1959,GurtinMacCamy1977,Vazquez2007}.
	
	In 1986, Li and Yau \cite{LiYau1986} applied the
	parabolic maximum principle to establish a gradient estimate for
	positive solutions of the heat equation on complete Riemannian manifolds with
	nonnegative Ricci curvature. More precisely, if $u$ is a positive solution of
	the heat equation, then
	\begin{equation}\label{eq:Li-Yau-intro}
		\frac{|\nabla u|^2}{u^2}-\frac{u_t}{u}
		\leq \frac{n}{2t}.
	\end{equation}
	By integrating this differential inequality along space--time curves, they
	derived the classical parabolic Harnack inequality and obtained heat-kernel
	estimates, together with several important geometric and analytic
	consequences; see also \cite{Davies1989,PeterLi}.
	
	In 1979, Aronson and B\'enilan
	\cite{AronsonBenilan1979} studied the nonlinear diffusion equation
	\begin{equation}\label{pme}
		u_t=\Delta u^m,
	\end{equation}
	in Euclidean space and introduced the pressure variable
	\(
	v_m=\frac{m}{m-1}u^{m-1},
	\)
	under which the equation \eqref{pme} becomes
	\[
	(v_m)_t=(m-1)v_m\Delta v_m+|\nabla v_m|^2.
	\]
	They obtained the Aronson--B\'enilan estimate
	\begin{equation}\label{eq:AB-intro}
		\Delta v_m\geq-\frac{\kappa_m}{t},
		\qquad
		\kappa_m=\frac{n}{n(m-1)+2},
		\qquad
		m>1-\frac{2}{n}.
	\end{equation}
	The Aronson--B\'enilan estimate is also a differential Harnack
	inequality for the pressure variable and plays an important role in the
	regularity and qualitative theory of the PME and FDE.
	
	In 2009, Lu, Ni, V\'azquez, and Villani \cite{LNVV2009} studied the PME and
	FDE on complete Riemannian manifolds and extended the Aronson--B\'enilan and
	Li--Yau estimates to this geometric setting. Their approach combines the
	pressure formulation with the Bochner formula and the localized parabolic
	maximum principle. More precisely, suppose that $m>1$,
	\(
	\Ric\geq-(n-1)k^2g,
	v_{\max}:=\sup_{M\times[0,T]}v_m<\infty.
	\)
	Then, for every $\beta>1$, they obtained the global estimate
	\begin{equation}\label{eq:LNVV-global-intro}
		\frac{|\nabla v_m|^2}{v_m}
		-\beta\frac{(v_m)_t}{v_m}
		\leq
		(m-1)\kappa_m\beta^2
		\left(
		\frac{1}{t}
		+\frac{(m-1)(n-1)}{\beta-1}k^2v_{\max}
		\right).
	\end{equation}
	They also established local differential Harnack estimates for the PME,
	corresponding estimates for the FDE in the range
	$1-\frac{2}{n}<m<1$, and Perelman-type entropy formulas. Further gradient
	estimates and entropy formulas for nonlinear diffusion equations on
	Riemannian manifolds can be found in
	\cite{HHL2013,HuangLi2014,Wang2018Entropy,Zhu2011,Zhu2013}.
	
	The first motivation is Yau's differential Harnack estimate for general
	filtration equations \cite{Yau1994}. Yau used the parabolic maximum principle
	under a quadratic condition on the constitutive function. We seek to replace
	the condition on $F$ by explicit bounds on $uF''/F'$.
	
	\bta[Yau {\cite[Theorem~6.1]{Yau1994}}]
	\label{thm:Yau-intro}
	Let $(M^n,g)$ be a compact Riemannian manifold with nonnegative Ricci curvature,
	and let $u$ be a smooth positive solution of
	\eqref{eq:general-equation-intro} with $F'>0$. Define $G$, up to an additive
	constant, by
	\(
	G'(r)=\frac{F'(r)}{r},
	\)
	and set $G=G(u)$ and $\widetilde G(\xi)=G(e^\xi)$. Fix a constant $\alpha$
	and a smooth function $c(t)$. Suppose that, for every $x\geq0$,
	\begin{align}
		&\left[
		(2-\alpha)(\alpha-1)\widetilde G''
		-\left(\frac{2}{n}\widetilde G'
		-\alpha\widetilde G''\right)(\alpha-1)^2
		\right]x^2
		+
		\left[
		(2-\alpha)\widetilde G''
		-(\alpha-1)\left(\frac{4}{n}\widetilde G'
		-2\alpha\widetilde G''\right)
		\right]c(t)x
		\notag\\
		&\quad-c'(t)
		-\left(\frac{2}{n}\widetilde G'
		-\alpha\widetilde G''\right)c(t)^2
		\leq0.
		\label{eq:Yau-quadratic-intro}
	\end{align}
	If $|\nabla G|^2-\alpha G_t-c(t)$ is smooth and nonpositive at the initial
	time, then
	\begin{equation}\label{eq:Yau-estimate-intro}
		|\nabla G|^2-\alpha G_t-c(t)\leq0.
	\end{equation}
	\eeta
	
	Hamilton's matrix estimate for the heat equation provides another
	foundational form of a differential Harnack inequality \cite{Hamilton1993}.
	Further maximum-principle estimates for filtration equations were obtained
	in \cite{MaZhaoSong2008,Xu2012,Wang2017Acta}. Related results for parabolic
	$p$-Laplacian equations appear in \cite{WangXue2021,WangXue2023}.
	
	The second motivation is Huang and Shen's use of Nash--Moser iteration for
	the PME and FDE \cite{HuangShen2025}.
	
	\bta[Huang--Shen {\cite{HuangShen2025}}]
	\label{thm:Huang-Shen-intro}
	Let $(M^n,g)$ be complete with $\Ric\geq-Kg$, $K\geq0$, and let $u>0$ be a
	weak solution of
	\[
	u_t=\Delta u^m,\qquad m>1,
	\]
	on $B_{2R}(x_0)\times(s-R^2,s+\varepsilon)$. Set
	\[
	v_m=\frac{m}{m-1}u^{m-1},
	\qquad
	a=\frac{n(m-1)}{n(m-1)+2},
	\qquad
	I_\tau=(s-\tau R^2,s].
	\]
	For every $\alpha>1$ and $0<\tau<1$, the positive $C^1$ function
	$\varphi_{0,\varsigma}$ specified in \cite{HuangShen2025} satisfies
	\begin{equation}\label{eq:Huang-Shen-PME-intro}
		\sup_{B_R(x_0)\times I_\tau}
		\left(
		\frac{|\nabla v_m|^2}{v_m}
		-\alpha\frac{(v_m)_t}{v_m}
		-\varphi_{0,\varsigma}
		\right)
		\leq
		\frac{a\alpha^2}{1+a(\alpha-1)}
		\frac{C}{(1-\tau)^{3+n}}
		\frac{1+\sqrt K R}{R^2},
	\end{equation}
	where $C$ depends on $m$, $n$, and the local upper and lower bounds of
	$v_m$. They also obtained the corresponding estimate for
	$1-\frac{2}{n}<m<1$ using the negative pressure variable.
	\eeta
	
	Motivated by these results, we use Nash--Moser iteration to study differential Harnack estimates for the general filtration equation \eqref{eq:general-equation-intro}.

	Once the Bochner differential inequality has been established in
	distributional form, Moser's iteration method \cite{Moser1964} uses a local
	Sobolev inequality and nested space--time cut-offs to convert integral energy
	estimates into pointwise bounds. The Bochner computation in this paper is
	carried out for smooth pressure solutions. After the pointwise Bochner
	calculation, the positive-part inequality is written in distributional form
	and iterated by energy estimates. This procedure does not require a maximum point
	for an auxiliary function and uses only first-order derivatives of the
	spatial cutoffs, which is convenient on complete noncompact manifolds. In
	the elliptic setting, Nash--Moser iteration has also been used to establish
	Cheng--Yau-type gradient estimates; see
	\cite{ChengYau1975,HeWangWei2024,HeHuWang2026,HeSunWang2024,
		WangHao2026,WangZhang2011}.
	
	We introduce the pressure variable
	\[
	v=v_*+\int_{u_*}^{u}\frac{F'(r)}{r}\,\dd r,
	\]
	where $u_*>0$ and $v_*\in\mathbb R$ are fixed. This transformation is
	strictly increasing. Writing $u=u(v)$ and setting
	\[
	h(v)=F'(u(v)),
	\]
	we obtain
	\[
	v_t=h(v)\Delta v+|\nabla v|^2,
	\qquad
	h'(v)=\frac{uF''(u)}{F'(u)}.
	\]
	
	Using Nash--Moser iteration, we obtain the following local differential Harnack estimate.
	
	\begin{theorem}[Local differential Harnack estimate]\label{thm:barrier-main}
		Let $(M^n,g)$ be a complete Riemannian manifold with
		\(
		\Ric\geq-(n-1)Kg
		\)
		for some $K\geq0$, where $n>2$. Let $B_{2R}=B_{2R}(x_0)\subset M$ and set
		\(
		\Q_{2R}=B_{2R}\times(s-R^2,s+\varepsilon).
		\)
		Let $u$ be a positive classical solution of \eqref{eq:general-equation-intro}
		on $\Q_{2R}$, where $F\in C^2(\mathcal I)$ and $F'>0$. Assume that the
		associated pressure variable $v$ is smooth on $\Q_{2R}$ and satisfies
		\eqref{eq:pressure-equation-intro} pointwise. Suppose that, for some $\alpha>2$ and
		constants $h_-,h_+,\delta_-$ and $\delta_+$,
		\begin{equation}\label{eq:structural-assumptions-main}
			0<h_-\leq h(v)=F'(u)\leq h_+<\infty,
			\qquad
			-\frac{2}{n\alpha}
			<\delta_-\leq h'(v)
			=\frac{uF''(u)}{F'(u)}
			\leq\delta_+<\infty
		\end{equation}
		on $\Q_{2R}$. Let
		$\phi\in C^1((s-R^2,s+\varepsilon))$ be nonnegative and satisfy
		\begin{equation}\label{eq:barrier-Riccati-main}
			\begin{aligned}
				\phi'
				+\frac{2+n\alpha\delta_-}{n\alpha^2h_+}\phi^2 
				-\frac{n\alpha^2h_+}
				{4(\alpha-1)\bigl(2(\alpha-1)+n\alpha\delta_-\bigr)}
				\left(
				2(n-1)Kh_+
				-\frac{4(\alpha-1)+n\alpha^2\delta_-}
				{n\alpha^2h_+}\phi
				\right)^2
				\geq0.
			\end{aligned}
		\end{equation}
		Then, for every $0<\tau<1$,
		\begin{equation}\label{eq:barrier-main-estimate}
			\sup_{B_R(x_0)\times I_\tau}
			\bigl(|\nabla v|^2-\alpha v_t-\phi(t)\bigr)
			\leq
			\frac{C}{(1-\tau)^{3+n/2}}
			\frac{1+\sqrt K R}{R^2},
		\end{equation}
		where $I_\tau=(s-\tau R^2,s]$, and $C$ depends only on
		$n,\alpha,h_-,h_+,\delta_-$ and $\delta_+$.
	\end{theorem}
	
	\begin{remark}\label{rem:structural-comments}
		\leavevmode
		\begin{itemize}
			\item For $F(u)=u^m$, one has $h'(v)=m-1$. Hence, for every fixed
			$m>1$, the structural condition holds without any upper restriction on
			$m$, although the constants may depend on $m$. For $0<m<1$, an admissible
			$\alpha>2$ exists precisely when $m>1-1/n$. This is the interior of
			McCann's displacement-convexity range for the corresponding R\'enyi
			internal energy \cite{McCann1997}. The endpoint is excluded because the
			coercive coefficient in the iteration vanishes there. The two-sided bounds
			on $h(v)$ restrict the result to the uniformly parabolic positive regime.
			
			\item Yau's condition is formulated as a quadratic inequality involving
			$\widetilde G'$, $\widetilde G''$, and a time-dependent function. Our
			condition is expressed directly through $uF''/F'$ and is simpler to verify.
			Its lower endpoint is strict because the coercive term vanishes there.
			
			\item The theorem also covers non-power diffusion laws with natural physical
			or geometric interpretations. The law $F(u)=au+bu^2$ models Brownian
			particles with short-range repulsive interactions
			\cite{BrunaChapmanRobinson2017}, while
			$F(u)=au+b[u-\log(1+u)]$ describes a saturating diffusivity. The concave
			law $F(u)=au+b\log(1+u)$ still admits a convex internal energy and a formal
			Wasserstein gradient-flow structure
			\cite{Otto2001,McCann1997,BolleyCarrillo2014}. The law
			$F(u)=Du/(1-cu)$ is related to volume-exclusion diffusion for Brownian hard
			rods \cite{GavishNyquistPeletier2023}. Further details are given in
			Section~\ref{sec:special-laws}.
		\end{itemize}
	\end{remark}
	
	The constant barrier $\phi_K$ used below is defined in
	\eqref{eq:explicit-barrier} after the coefficients in the pointwise
	inequality have been introduced.
	
	\begin{theorem}[Global differential Harnack estimate]
		\label{thm:global}
		Let $(M^n,g)$ be a complete Riemannian manifold, $n>2$, satisfying
		\(
		\Ric\geq-(n-1)Kg
		\)
		for some $K\geq0$. Let $u$ be a positive ancient classical solution of
		\eqref{eq:general-equation-intro} on $M\times(-\infty,T)$, where
		$F\in C^2(\mathcal I)$ and $F'>0$. Assume that the associated pressure
		variable $v$ is smooth on $M\times(-\infty,T)$ and satisfies
		\eqref{eq:pressure-equation-intro} pointwise. Suppose that,
		for some $\alpha>2$ and constants $h_-,h_+,\delta_-$ and $\delta_+$,
		\[
		0<h_-\leq h(v)\leq h_+<\infty
		\]
		and
		\[
		-\frac{2}{n\alpha}
		<\delta_-
		\leq h'(v)
		\leq\delta_+
		<\infty
		\]
		on $M\times(-\infty,T)$. Then
		\begin{equation}\label{eq:global-Harnack-estimate}
			|\nabla v|^2-\alpha v_t\leq\phi_K
			\qquad\text{on }M\times(-\infty,T).
		\end{equation}
	\end{theorem}
	
	\begin{theorem}[Harnack inequality]\label{thm:Harnack-additive}
		Under the assumptions of Theorem~\ref{thm:barrier-main}, take
		$\phi=\phi_K$ and write
		\[
		E_{R,\tau}:=
		\frac{C(n,\alpha,h_-,h_+,\delta_-,\delta_+)}
		{(1-\tau)^{3+n/2}}
		\frac{1+\sqrt K R}{R^2}.
		\]
		If $t_1<t_2$ belong to $I_\tau$ and a minimizing geodesic from $x_1$ to
		$x_2$ is contained in $B_R(x_0)$, then
		\begin{equation}\label{eq:additive-Harnack}
			v(x_1,t_1)\le v(x_2,t_2)
			+\frac{\alpha d(x_1,x_2)^2}{4(t_2-t_1)}
			+\frac{\phi_K+E_{R,\tau}}{\alpha}(t_2-t_1).
		\end{equation}
	\end{theorem}
	
	\begin{theorem}[Liouville-type theorem]\label{thm:Liouville}
		Assume the hypotheses of Theorem~\ref{thm:global}, with $K=0$ and $M$
		connected.  If $u_t\le0$, then $u$ is constant on space--time.
	\end{theorem}
	
	The differential Harnack and frequency results use two complementary
	structures generated by the same filtration law. If $U''(u)=F'(u)/u$, then
	the pressure is $v=U'(u)$ up to a constant; the frequency entropy instead
	satisfies $\Psi_a''(u)=F'(u)$. The former is suited to the local Bochner
	calculation, while the latter produces global dissipation identities without
	a curvature assumption.

	The frequency-function method was introduced by Almgren
	\cite{Almgren1979} for elliptic equations and was further developed by
	Garofalo and Lin \cite{GarofaloLin1986} in the study of unique continuation.
	Lin \cite{Lin1990} extended this method to parabolic equations, and Poon
	\cite{Poon1996} established a frequency monotonicity formula for the heat
	equation. More precisely, if $w_t=\Delta w$ and $H(x,x_0;t)$ denotes the
	heat kernel, the parabolic frequency is defined by
	\begin{equation}\label{eq:parabolic-frequency-intro}
		\mathcal N_{\mathrm P}(t)
		=
		t\,
		\frac{\displaystyle
			\int_M H(x,x_0;t)|\nabla w|^2(x,T-t)\,\dd\mu}
		{\displaystyle
			\int_M H(x,x_0;t)w^2(x,T-t)\,\dd\mu}.
	\end{equation}
	Poon proved its monotonicity in Euclidean space and obtained strong unique
	continuation for parabolic equations. On a compact manifold with
	nonnegative sectional curvature, Li and Wang \cite{LiWang2019} proved that
	\[
	e^{\sqrt t}\mathcal N_{\mathrm P}(t)
	\]
	is nondecreasing for sufficiently small $t$, which gives an
	almost-monotonicity formula and a corresponding unique continuation result.
	
	Colding and Minicozzi \cite{ColdingMinicozzi2022} considered the drift
	Laplacian
	\[
	\mathcal L_\psi
	=\Delta-\langle\nabla\psi,\nabla\cdot\rangle
	\]
	on a weighted Riemannian manifold. For a solution of
	$w_t=\mathcal L_\psi w$, they defined
	\[
	I_\psi(t)=\int_Mw^2e^{-\psi}\,\dd\mu,
	\qquad
	D_\psi(t)=-\int_M|\nabla w|^2e^{-\psi}\,\dd\mu,
	\qquad
	\mathcal U_\psi(t)=\frac{D_\psi(t)}{I_\psi(t)}.
	\]
	They proved, without any curvature assumption, that
	\begin{equation}\label{eq:CM-frequency-intro}
		\mathcal U_\psi'(t)
		=
		\frac{2}{I_\psi(t)^2}
		\left[
		I_\psi(t)\int_M|\mathcal L_\psi w|^2e^{-\psi}\,\dd\mu
		-D_\psi(t)^2
		\right]
		\geq0.
	\end{equation}
	This monotonicity yields parabolic unique continuation and backward
	uniqueness. Baldauf and Kim \cite{BaldaufKim2023} later extended the
	frequency method to parabolic equations along Ricci flows.
	
	The functional used below is not the heat-kernel frequency
	\eqref{eq:parabolic-frequency-intro}: it contains neither a base point nor a
	parabolic scaling factor. It is an entropy-normalized dissipation, or a
	nonlinear Rayleigh quotient. We retain the term frequency because this
	terminology is standard for the corresponding nonlinear diffusion
	quotients; its proof uses global energy identities rather than vanishing-order
	or unique-continuation methods.
	
	Liu and Xu \cite{LiuXu2022} proved frequency monotonicity for the weighted
	$p$-Laplacian heat equation. For $p=2$ and
	constant weight, their quotient reduces, up to its sign and a factor of two,
	to the quantity below when $F(u)=u$ and $a=0$.
	
	More recently, Sun and S\"urig \cite{SunSurig2026} studied
	\[
	\partial_t w=\mathcal L_{p,\psi}(w^q)
	\]
	on weighted complete Riemannian manifolds. They introduced
	\begin{equation}\label{eq:Sun-Surig-frequency-intro}
		\mathcal N_G(t)
		=
		\frac{-\displaystyle\int_M|\nabla(w^q)|^p e^{-\psi}\,\dd\mu}
		{\left(\displaystyle\int_M|w|^{q+1}e^{-\psi}\,\dd\mu
			\right)^{pq/(q+1)}}
	\end{equation}
	and proved that $\mathcal N_G'(t)\geq0$. When $p=2$, $q=m$, and $\psi$ is
	constant, their equation contains both the PME and the FDE. Thus the power-law
	formula below recovers the corresponding homogeneous frequency. Our extension
	is the use of the relative entropy determined by $F$ to treat general,
	nonhomogeneous filtration laws. Related gradient and entropy estimates for
	doubly nonlinear and $p$-filtration equations appear in
	\cite{WangChen2014,WangHaoUnpublished}.
	
	Fix a reference state $a$ and set
	\begin{equation}\label{eq:frequency-data-intro}
		G_a(s)=F(s)-F(a),\qquad
		\Psi_a(s)=\int_a^sG_a(r)\,\dd r.
	\end{equation}
	For a frequency-admissible solution in the sense of
	Definition~\ref{def:frequency-admissible}, define
	\begin{equation}\label{eq:frequency-energies-intro}
		\mathcal E_a(t)=\int_M\Psi_a(u)\,\dd\mu,\qquad
		\mathcal D(t)=\int_M|\nabla F(u)|^2\,\dd\mu.
	\end{equation}
	The identities $\mathcal E_a'=-\mathcal D$ and
	$\mathcal D'=-2\int_M F'(u)(\Delta F(u))^2\,\dd\mu$ lead to our second main
	result.
	
	\begin{theorem}[Global entropy--dissipation frequency monotonicity]
		\label{thm:frequency-main}
		Let $u$ be a frequency-admissible solution of
		\eqref{eq:general-equation-intro} on a complete Riemannian manifold, relative
		to a reference state $a$, and suppose that $\mathcal E_a(t)>0$ for every
		$t\in(0,T)$. Assume that,
		on the range of $u$,
		\begin{equation}\label{eq:frequency-structure-intro}
			\bigl(F(s)-F(a)\bigr)^2
			\le \kappa F'(s)\int_a^s\bigl(F(r)-F(a)\bigr)\,\dd r
		\end{equation}
		for some $\kappa>0$. Then, for every $0<\gamma\le2/\kappa$,
		\begin{equation}\label{eq:frequency-main-intro}
			\mathcal N_\gamma(t)
			:=\frac{\mathcal D(t)}{\mathcal E_a(t)^\gamma}
		\end{equation}
		is nonincreasing. No curvature sign is required for this global conclusion.
	\end{theorem}
	
	The same monotonicity also implies that a suitable power transform of
	$\mathcal E_a$, with $\log\mathcal E_a$ at $\gamma=1$, is convex in time.
	Corollary~\ref{cor:frequency-convexity-decay} gives the resulting three-time
	inequality and quantitative lower bounds for the decay of the entropy.
	
	\begin{remark}\label{rem:frequency-structure-automatic}
		The structural condition is automatic on compact solution ranges in the
		following sense. Define
		\[
		\mathscr R_a(s)=
		\frac{(F(s)-F(a))^2}{F'(s)\Psi_a(s)}.
		\]
		If $F'$ is positive on a compact interval containing the range of $u$ and
		$F'(a)>0$, then $\mathscr R_a$ extends continuously to $s=a$, with
		$\mathscr R_a(a)=2$. Hence $\mathscr R_a$ is bounded there, and one may
		take $\kappa=\sup\mathscr R_a$. The theorem therefore imposes a genuine
		restriction mainly when the solution range approaches an endpoint at which
		$F'$ degenerates or becomes singular.
	\end{remark}
	
	For $F(s)=s^m$ and $a=0$, the structural ratio in
	\eqref{eq:frequency-structure-intro} equals $(m+1)/m$. Consequently,
	\begin{equation}\label{eq:power-frequency-intro}
		t\longmapsto
		\frac{\displaystyle\int_M|\nabla u^m|^2\,\dd\mu}
		{\displaystyle\left(\int_M u^{m+1}\,\dd\mu\right)^{2m/(m+1)}}
	\end{equation}
	is nonincreasing for every $m>0$ whenever the required finite-energy
	assumptions hold. Up to the sign convention and a constant factor, this is the
	$p=2$ frequency in \cite{SunSurig2026}. The restriction $m>1-1/n$ belongs
	only to the differential Harnack theorem and is not needed for this frequency
	formula.
	
	Section~\ref{sec:prelim} derives the pressure equation and the Bochner
	inequalities. Section~\ref{sec:proof} proves the local differential Harnack
	estimate and its consequences. Section~\ref{sec:frequency} proves
	Theorem~\ref{thm:frequency-main}. Section~\ref{sec:special-laws} discusses
	power and non-power filtration laws and their geometric and physical
	interpretations.
	
	\section{Preliminaries}\label{sec:prelim}
	
	The equation
	\[
	u_t=\Delta F(u)
	=\operatorname{div}\bigl(F'(u)\nabla u\bigr)
	\]
	is a quasilinear parabolic equation in divergence form. On a space--time
	cylinder $Q$, its uniform parabolicity condition is
	\begin{equation}\label{eq:uniform-parabolic-u}
		0<h_-\leq F'(u)\leq h_+<\infty
		\qquad\text{a.e. on }Q.
	\end{equation}
	The standard weak-solution theory and the corresponding existence and
	interior regularity results for Cauchy and initial-boundary value problems
	can be found in
	\cite{LadyzhenskayaSolonnikovUraltseva1968,AltLuckhaus1983,
		DiBenedetto1993,Vazquez2007}. Under
	\eqref{eq:uniform-parabolic-u}, locally bounded weak solutions are locally
	H\"older continuous. Further interior regularity requires corresponding
	smoothness assumptions on $F$ and the data. The Bochner computations below
	are made for solutions whose pressure variable is smooth.
	
	We now introduce the pressure transformation.
	
	\begin{proposition}\label{prop:explicit-derivatives}
	Let $\mathcal I\subset(0,\infty)$ be open, let
	$F\in C^2(\mathcal I)$ with $F'>0$, and let $u$ be a classical solution
	of \eqref{eq:general-equation-intro} with values in $\mathcal I$. Fix
	$u_*\in\mathcal I$ and $v_*\in\mathbb R$, and define
	\[
	G(r)=v_*+\int_{u_*}^{r}\frac{F'(\xi)}{\xi}\,\dd\xi,
	\qquad
	v=G(u).
	\]
	Then $G$ is strictly increasing. If $J=G(\mathcal I)$, we may write
	$u=G^{-1}(v)$ and define
	\[
	h(v)=F'(G^{-1}(v)),
	\qquad v\in J.
	\]
	The pressure variable satisfies
	\begin{equation}\label{eq:pressure-equation-intro}
		v_t=h(v)\Delta v+|\nabla v|^2.
	\end{equation}
	Moreover,
	\begin{equation}\label{eq:pressure-derivatives}
		\frac{\dd v}{\dd u}=\frac{F'(u)}{u},
		\qquad
		\nabla v=\frac{F'(u)}{u}\nabla u,
		\qquad
		v_t=\frac{F'(u)}{u}u_t,
		\qquad
		\frac{\dd u}{\dd v}=\frac{u}{F'(u)},
	\end{equation}
	and
	\begin{equation}\label{eq:hv-F}
		h'(v)=\frac{uF''(u)}{F'(u)}.
	\end{equation}
	\end{proposition}
	
	\begin{proof}
		Since $G'(u)=F'(u)/u>0$, the function $G$ is strictly increasing and
		has an inverse on $J$. The identities in
		\eqref{eq:pressure-derivatives} follow from the chain rule. Moreover,
		\begin{equation}\label{eq:Delta-v-in-u}
			\Delta v
			=
			\frac{F'(u)}{u}\Delta u
			+
			\left(
			\frac{F''(u)}{u}
			-\frac{F'(u)}{u^2}
			\right)|\nabla u|^2.
		\end{equation}
		Since
		\[
		u_t=\Delta F(u)
		=F'(u)\Delta u+F''(u)|\nabla u|^2,
		\]
		we obtain
		\begin{align*}
			v_t
			=
			\frac{F'(u)^2}{u}\Delta u
			+\frac{F'(u)F''(u)}{u}|\nabla u|^2
			=
			F'(u)\Delta v+|\nabla v|^2,
		\end{align*}
		which proves \eqref{eq:pressure-equation-intro}. Finally,
		\[
		h'(v)
		=
		F''(u)\frac{\dd u}{\dd v}
		=
		\frac{uF''(u)}{F'(u)}.
		\]
	\end{proof}
	
	Since $h(v)=F'(u)$, the uniform parabolicity condition for the pressure
	equation is precisely
	\begin{equation}\label{eq:uniform-parabolic-pressure}
		0<h_-\leq h(v)\leq h_+<\infty.
	\end{equation}
	
	\begin{definition}[Locally uniformly parabolic weak pressure solution]
		\label{def:weak-pressure}
		Let $\Omega\subset M$ be open, let $I=(t_0,t_1)$, and set
		$Q=\Omega\times I$. A function
		\[
		v\in L^\infty_{\mathrm{loc}}(Q)
		\cap
		L^2_{\mathrm{loc}}
		\bigl(I;W^{1,2}_{\mathrm{loc}}(\Omega)\bigr)
		\]
		is called a locally uniformly parabolic weak solution of
		\eqref{eq:pressure-equation-intro} if $v(x,t)\in J$ almost everywhere
		and, for every $Q'\Subset Q$, there exist constants
		$h_-(Q')$, $h_+(Q')$, and $L(Q')$ such that
		\begin{equation}\label{eq:local-uniform-pressure}
			0<h_-(Q')
			\leq h(v)
			\leq h_+(Q')<\infty,
			\qquad
			|h'(v)|\leq L(Q')
			\quad\text{a.e. on }Q'.
		\end{equation}
		It is also required that
		\begin{equation}\label{eq:weak-pressure}
			\int_Q
			\left[
			-v\psi_t
			+h(v)\langle\nabla v,\nabla\psi\rangle
			\right]\dd\mu\dd t
			=
			\int_Q
			(1-h'(v))|\nabla v|^2\psi\,\dd\mu\dd t
		\end{equation}
		for every $\psi\in C_c^\infty(Q)$.
	\end{definition}
	
	Indeed,
	\[
	\operatorname{div}\bigl(h(v)\nabla v\bigr)
	=h(v)\Delta v+h'(v)|\nabla v|^2,
	\]
	so \eqref{eq:weak-pressure} is the distributional formulation of
	\[
	v_t-\operatorname{div}\bigl(h(v)\nabla v\bigr)
	=(1-h'(v))|\nabla v|^2.
	\]
	The two-sided bound for $h(v)$ gives local uniform parabolicity, while the
	bound for $h'(v)$ ensures that the right-hand side belongs to
	$L^1_{\mathrm{loc}}(Q)$. If the range of $v$ on $Q'$ is contained in a
	compact interval $J'\Subset J$, both bounds follow from the continuity of
	$h$ and $h'$. This definition records the natural energy formulation of the
	pressure equation. The differential Harnack theorems in this paper are stated
	for smooth pressure solutions; extending them to all weak solutions in
	Definition~\ref{def:weak-pressure} would require a separate regularization
	and stability argument.
	
	The next lemma is the local Sobolev input for the iteration; standard
	background on Sobolev inequalities on manifolds can be found in
	\cite{PeterLi}.
	\begin{lemma}[Local Sobolev inequality {\cite[Theorem~3.1]{SaloffCoste1992}}]
		\label{lem:local-Sobolev}
		Suppose that $n>2$ and $\Ric\ge-(n-1)Kg$ for some $K\ge0$.
		Then there exists a constant $C_S=C_S(n)>0$ such that, for every
		$\zeta\in C_0^\infty(B_{2R}(x_0))$,
		\begin{equation}\label{eq:local-Sobolev}
			\left(
			\int_{B_{2R}}|\zeta|^{2n/(n-2)}\dd\mu
			\right)^{(n-2)/n}
			\leq
			\mathcal S_R
			\int_{B_{2R}}
			\bigl(|\nabla\zeta|^2+R^{-2}\zeta^2\bigr)\dd\mu,
		\end{equation}
		where
		\begin{equation}\label{eq:S_R}
			\mathcal S_R
			=C_S\exp\!\bigl(C_S(1+\sqrt K R)\bigr)
			R^2\vol(B_{2R})^{-2/n}.
		\end{equation}
		Let $\chi=1+2/n$.  Then, for every time interval $I$ and every smooth
		function $\zeta$ with compact spatial support in $B_{2R}(x_0)$,
		\begin{align}
			\left(
			\int_I\int_{B_{2R}}|\zeta|^{2\chi}\dd\mu\dd t
			\right)^{1/\chi}
			\leq{}
			\mathcal S_R^{1/\chi}
			\left(
			\sup_{t\in I}
			\int_{B_{2R}}\zeta^2\dd\mu
			\right)^{2/(n\chi)}
			\left(
			\int_I\int_{B_{2R}}
			\bigl(|\nabla\zeta|^2+R^{-2}\zeta^2\bigr)
			\dd\mu\dd t
			\right)^{1/\chi}.
			\label{eq:parabolic-Sobolev}
		\end{align}
	\end{lemma}
	
	We next use the Bochner formula to derive the differential identities needed
	below. Define the parabolic operator
	\begin{equation}\label{eq:L-operator}
		\cL:=\partial_t-h(v)\Delta
	\end{equation}
	and set
	\begin{equation}\label{eq:w-y-z-H}
		y=|\nabla v|^2,
		\qquad
		z=v_t,
		\qquad
		H_\alpha=y-\alpha z.
	\end{equation}
	By \eqref{eq:pressure-equation-intro},
	\[
	z=h(v)\Delta v+y.
	\]
	Therefore,
	\begin{equation}\label{eq:Delta-v-H}
		z=\frac{y-H_\alpha}{\alpha},
		\qquad
		H_\alpha=-\alpha h(v)\Delta v-(\alpha-1)y,
		\qquad
		\Delta v
		=-\frac{(\alpha-1)y+H_\alpha}{\alpha h(v)}.
	\end{equation}
	
	\begin{lemma}\label{lem:evolution-identities}
		We have
		\begin{align}
			\cL y
			={}&2\langle\nabla v,\nabla y\rangle
			+2h'(v)y\Delta v
			-2h(v)\bigl(|\nabla^2v|^2+\Ric(\nabla v,\nabla v)\bigr),
			\label{eq:Ly}\\
			\cL z
			={}&2\langle\nabla v,\nabla z\rangle
			+h'(v)h(v)(\Delta v)^2+h'(v)y\Delta v,
			\label{eq:Lz}\\
			\cL H_\alpha
			={}&2\langle\nabla v,\nabla H_\alpha\rangle
			-2h(v)\bigl(|\nabla^2v|^2+\Ric(\nabla v,\nabla v)\bigr)
			-\alpha h'(v)h(v)(\Delta v)^2
			-(\alpha-2)h'(v)y\Delta v.
			\label{eq:LH-exact}
		\end{align}
	\end{lemma}
	
	\begin{proof}
		The Bochner formula gives
		\[
		\Delta y=2|\nabla^2v|^2+2\Ric(\nabla v,\nabla v)
		+2\langle\nabla v,\nabla\Delta v\rangle.
		\]
		Differentiating the pressure equation $v_t=h(v)\Delta v+y$ in the spatial
		variables yields
		\[
		\nabla v_t=h(v)\nabla\Delta v+h'(v)(\Delta v)\nabla v+\nabla y.
		\]
		Consequently,
		\begin{align*}
			\cL y
			&=2\langle\nabla v,\nabla v_t\rangle-h(v)\Delta y\\
			&=2\langle\nabla v,\nabla y\rangle+2h'(v)y\Delta v
			-2h(v)\bigl(|\nabla^2v|^2+\Ric(\nabla v,\nabla v)\bigr).
		\end{align*}
		This proves \eqref{eq:Ly}.
		
		Differentiating $v_t=h(v)\Delta v+y$ in time gives
		\[
		v_{tt}=h'(v)v_t\Delta v+h(v)\Delta v_t
		+2\langle\nabla v,\nabla v_t\rangle.
		\]
		Therefore
		\begin{align*}
			\cL z=\cL(v_t)=h'(v)z\Delta v
			+2\langle\nabla v,\nabla z\rangle
			=2\langle\nabla v,\nabla z\rangle
			+h'(v)h(v)(\Delta v)^2+h'(v)y\Delta v,
		\end{align*}
		where we used $z=h(v)\Delta v+y$.  This proves \eqref{eq:Lz}.
		
		Finally, $H_\alpha=y-\alpha z$, and hence
		\begin{align*}
			\cL H_\alpha
			={}&\cL y-\alpha\cL z\\
			={}&2\langle\nabla v,\nabla H_\alpha\rangle
			-2h(v)\bigl(|\nabla^2v|^2+\Ric(\nabla v,\nabla v)\bigr)
			+h'(v)(2y-\alpha z)\Delta v\\
			={}&2\langle\nabla v,\nabla H_\alpha\rangle
			-2h(v)\bigl(|\nabla^2v|^2+\Ric(\nabla v,\nabla v)\bigr)
			-\alpha h'(v)h(v)(\Delta v)^2
			-(\alpha-2)h'(v)y\Delta v.
		\end{align*}
		This proves \eqref{eq:LH-exact}.
	\end{proof}
	
	\begin{lemma}\label{lem:pointwise}
		Let $v$ be a smooth solution of \eqref{eq:pressure-equation-intro}.
		Assume that $n>2$, $\Ric\ge-(n-1)Kg$, and
		\begin{equation}\label{eq:h-bounds-pointwise}
			0<h_-\leq h(v)\leq h_+<\infty
			\qquad\text{on }\Q_{2R}.
		\end{equation}
		Let $\alpha>2$ and assume that
		\begin{equation}\label{eq:delta-minus}
			-\frac{2}{n\alpha}<\delta_-\le h'(v)\le\delta_+<\infty
			\qquad\text{on }\Q_{2R}.
		\end{equation}
		Define
		\begin{align}
			c_0=\frac{2+n\alpha\delta_-}{n\alpha^2h_+},\notag\qquad
			d_0=\frac{4(\alpha-1)+n\alpha^2\delta_-}
			{n\alpha^2h_+},\notag\qquad
			e_0=\frac{(\alpha-1)\bigl(2(\alpha-1)+n\alpha\delta_-\bigr)}
			{n\alpha^2h_+}.
			\label{eq:coercive-constants}
		\end{align}
		Then $c_0,d_0,e_0>0$.  Let $\phi=\phi(t)\ge0$ be a $C^1$ function
		satisfying
		\begin{equation}\label{eq:new-barrier-condition}
			\phi'(t)+c_0\phi(t)^2
			-\frac{1}{4e_0}
			\left(2(n-1)Kh_+-d_0\phi(t)\right)^2\ge0.
		\end{equation}
		If
		\begin{equation}\label{eq:f-positive-part}
			f=(H_\alpha-\phi)_+,
		\end{equation}
		then
		\begin{equation}\label{eq:pointwise-coercive}
			\cL f\le
			2\langle\nabla v,\nabla f\rangle-c_0f^2-d_0yf
		\end{equation}
		in the sense of distributions on $\Q_{2R}$.
	\end{lemma}
	
	\begin{proof}
		On the open set $\Omega_+=\{H_\alpha>\phi\}$, one has
		$H_\alpha=f+\phi$, $\nabla H_\alpha=\nabla f$, and
		$\cL f=\cL H_\alpha-\phi'$.  Since
		\[
		|\nabla^2v|^2\ge\frac1n(\Delta v)^2,
		\qquad
		\Ric(\nabla v,\nabla v)\ge-(n-1)Ky,
		\]
		formula \eqref{eq:LH-exact} gives
		\begin{align}
			\cL f\le{}&2\langle\nabla v,\nabla f\rangle
			-\left(\frac{2h(v)}n+\alpha h'(v)h(v)\right)(\Delta v)^2\notag\\
			&-(\alpha-2)h'(v)y\Delta v
			+2(n-1)Kh(v)y-\phi'.
			\label{eq:Lf-before-expansion}
		\end{align}
		On $\Omega_+$, \eqref{eq:Delta-v-H} becomes
		\begin{equation}\label{eq:z-Delta-on-positive-set}
			z=\frac{y-f-\phi}{\alpha},
			\qquad
			\Delta v=-\frac{(\alpha-1)y+f+\phi}{\alpha h(v)}.
		\end{equation}
		For brevity, write
		\begin{align}
			A(v)=\frac{2+n\alpha h'(v)}{n\alpha^2h(v)},\qquad
			B(v)=\frac{4(\alpha-1)+n\alpha^2h'(v)}
			{n\alpha^2h(v)},\qquad
			C(v)=\frac{(\alpha-1)\bigl(2(\alpha-1)+n\alpha h'(v)\bigr)}
			{n\alpha^2h(v)}.
			\label{eq:ABC-coefficients}
		\end{align}
		Substitution of \eqref{eq:z-Delta-on-positive-set} into
		\eqref{eq:Lf-before-expansion}, followed by expansion of the square,
		yields
		\begin{align}
			\cL f\le{}&2\langle\nabla v,\nabla f\rangle
			-A(v)f^2-B(v)yf-2A(v)f\phi-C(v)y^2\notag\\
			&-B(v)y\phi-A(v)\phi^2
			+2(n-1)Kh(v)y-\phi'.
			\label{eq:Lf-full-expansion}
		\end{align}
		This is the expansion in which all coefficients have been collected.
		The assumption \eqref{eq:delta-minus}, together with $\alpha>2$ and
		$h(v)\le h_+$, implies
		\begin{equation}\label{eq:ABC-lower-bounds}
			A(v)\ge c_0,
			\qquad B(v)\ge d_0,
			\qquad C(v)\ge e_0>0.
		\end{equation}
		Because $f,y,\phi\ge0$, we may discard the term $-2A(v)f\phi$.
		The terms that remain after $-A(v)f^2-B(v)yf$ are bounded above by
		\begin{equation}\label{eq:quadratic-in-y}
			\mathscr Q_\phi(y)
			=-e_0y^2+\bigl(2(n-1)Kh_+-d_0\phi\bigr)y
			-c_0\phi^2-\phi'.
		\end{equation}
		The discriminant of this quadratic polynomial is
		\[
		\bigl(2(n-1)Kh_+-d_0\phi\bigr)^2
		-4e_0\bigl(c_0\phi^2+\phi'\bigr).
		\]
		By \eqref{eq:new-barrier-condition}, the discriminant is nonpositive.
		Since the leading coefficient of $\mathscr Q_\phi$ is $-e_0<0$, it follows
		that $\mathscr Q_\phi(y)\le0$ for every $y\in\mathbb R$.
		Thus \eqref{eq:pointwise-coercive} holds pointwise on $\Omega_+$.
		To justify the distributional statement, set $g=H_\alpha-\phi$ and
		let $\rho_\varepsilon$ be a smooth convex approximation of the
		positive-part function such that
		$0\leq\rho_\varepsilon'\leq1$ and
		$\rho_\varepsilon''\geq0$. Since $h(v)>0$,
		\[
		\cL\rho_\varepsilon(g)
		=\rho_\varepsilon'(g)\cL g
		-h(v)\rho_\varepsilon''(g)|\nabla g|^2
		\leq\rho_\varepsilon'(g)\cL g.
		\]
		Letting $\varepsilon\downarrow0$ gives
		\[
		\cL g_+
		\leq\mathbf 1_{\{g>0\}}\cL g
		\]
		in the sense of distributions. Moreover,
		\[
		g_+=\mathbf 1_{\{g>0\}}g,
		\qquad
		\nabla g_+=\mathbf 1_{\{g>0\}}\nabla g
		\quad\text{a.e.}
		\]
		Hence the pointwise estimate on $\Omega_+$ yields, distributionally,
		\begin{align*}
			\cL g_+
			&\leq\mathbf 1_{\{g>0\}}
			\bigl(2\langle\nabla v,\nabla g\rangle-c_0g^2-d_0yg\bigr)\\
			&=2\langle\nabla v,\nabla g_+\rangle-c_0g_+^2-d_0yg_+.
		\end{align*}
		Since $f=g_+$, this is precisely \eqref{eq:pointwise-coercive}.
	\end{proof}
	
	\begin{remark}\label{rem:alpha-range-new-H}
		The restriction $\alpha>2$ reduces the simultaneous positivity of
		$A$, $B$, and $C$ to the single condition
		$h'(v)>-2/(n\alpha)$.  If $1<\alpha<2$, the same calculation remains valid
		under the stronger condition
		\[
		h'(v)>-\frac{2(\alpha-1)}{n\alpha}.
		\]
	\end{remark}
	
	The constant function
	\begin{equation}\label{eq:explicit-barrier}
		\phi_K:=\frac{2(n-1)Kh_+}{d_0+2\sqrt{c_0e_0}}
	\end{equation}
	satisfies
	\eqref{eq:new-barrier-condition} with equality.  Alternatively,
	one may require
	\[
	2(n-1)Kh_+-d_0\phi\le0,
	\qquad \phi'+c_0\phi^2\ge0.
	\]
	These conditions place the vertex of \eqref{eq:quadratic-in-y} in the
	nonpositive half-line and make $\mathscr Q_\phi(0)\le0$.  For a constant
	barrier they give $\phi=2(n-1)Kh_+/d_0$, which is larger than
	\eqref{eq:explicit-barrier}; we therefore use the discriminant condition.

	\section{Proof of the Differential Harnack Estimates}\label{sec:proof}
	
	Throughout this section, all space--time integrals are taken with respect to
	$\dd\mu\dd t$. We assume the uniform parabolicity bounds
	\[
	0<h_-\leq h(v)\leq h_+<\infty
	\]
	and the bounds \eqref{eq:delta-minus} for $h'$. Set
	\begin{equation}\label{eq:hprime-upper-energy}
		\Lambda:=\max\bigl\{|2-\delta_-|,|2-\delta_+|\bigr\}.
	\end{equation}
	Then $|2-h'(v)|\le\Lambda$ on $\Q_{2R}$. The upper structural constant
	$\delta_+$ enters the proof only through this drift bound.  In the abstract
	energy estimates below, let $c_0,d_0>0$ and put
	\[
	\mathscr D=(n,h_-,h_+,c_0,d_0,\Lambda).
	\]
	The values of $c_0$ and $d_0$ furnished by the Bochner calculation will be
	inserted at the end of the section.
	
	\subsection{Energy inequalities}
	
	\begin{lemma}\label{lem:weak}
		Suppose that $f\geq0$ is locally Lipschitz and satisfies
		\[
		\cL f\le2\langle\nabla v,\nabla f\rangle
		-c_0f^2-d_0yf
		\]
		in distributions, where $c_0,d_0>0$.  Then, for every nonnegative
		$\psi\in C_0^\infty(\Q_{2R})$,
		\begin{align}
			\int\!\!\int
			\bigl(\psi f_t+h\langle\nabla\psi,\nabla f\rangle\bigr)
			\leq
			\int\!\!\int
			\Bigl((2-h'(v))\psi\langle\nabla v,\nabla f\rangle
			-c_0\psi f^2-d_0\psi yf\Bigr).
			\label{eq:weak-form}
		\end{align}
	\end{lemma}
	
	\begin{proof}
		Let $\psi\geq0$ be smooth and compactly supported.  Multiply the assumed weak differential inequality by $\psi$ and integrate.  The
		second-order term is treated by spatial integration by parts:
		\begin{align*}
			\int\!\!\int\psi\cL f
			&=\int\!\!\int\psi f_t-\int\!\!\int\psi h\Delta f\\
			&=\int\!\!\int\psi f_t
			+\int\!\!\int\langle\nabla(h\psi),\nabla f\rangle\\
			&=\int\!\!\int\psi f_t
			+\int\!\!\int h\langle\nabla\psi,\nabla f\rangle
			+\int\!\!\int\psi\langle\nabla h,\nabla f\rangle.
		\end{align*}
		Since $\nabla h=h'(v)\nabla v$, the last term equals
		\[
		\int\!\!\int h'(v)\psi\langle\nabla v,\nabla f\rangle.
		\]
		Moving this contribution to the right changes the coefficient of
		$\psi\langle\nabla v,\nabla f\rangle$ from $2$ to $2-h'(v)$.
		This proves \eqref{eq:weak-form}.
	\end{proof}
	
	\begin{lemma}\label{lem:Caccioppoli}
		Under the assumptions of Lemma~\ref{lem:weak}, let $b\geq1$.
		Let $J=(t_0,t_1]$, and let $\lambda=\lambda(t)$ and
		$\eta\in C_c^\infty(B_{2R})$ be smooth cutoffs satisfying
		$0\leq\lambda,\eta\leq1$ and $\lambda(t_0)=0$. Set
		\begin{equation}\label{eq:b-threshold}
			\Lambda:=\sup_{\Q_{2R}}|2-h'(v)|<\infty,
			\qquad b_*:=\max\left\{1,\frac{2\Lambda^2}{d_0h_-}\right\}.
		\end{equation}
		If $b\geq b_*$
		then
		\begin{align}
			&\frac1{b+1}\sup_{t\in J}\int \lambda\eta^2f^{b+1}\dd\mu
			+\frac b2\int_J\!\!\int h\lambda\eta^2f^{b-1}|\nabla f|^2
			+c_0\int_J\!\!\int\lambda\eta^2f^{b+2}\notag\\
			&\qquad
			+\frac{d_0}{2}\int_J\!\!\int\lambda\eta^2yf^{b+1}
			\leq
			2\int_J\!\!\int
			\left(\frac{|\lambda'|}{b+1}\eta^2
			+\frac{4h_+}{b}\lambda|\nabla\eta|^2\right)f^{b+1}.
			\label{eq:Caccioppoli}
		\end{align}
	\end{lemma}
	
	\begin{proof}
		For a rigorous use of this $f$-dependent test, replace $f$ by
		$f_\varepsilon=f+\varepsilon$, take a Steklov average in time, and insert a
		smooth terminal cutoff converging to the indicator of $(t_0,t]$.  Perform
		the calculation for the resulting admissible test function, then pass in
		the following order: first let the Steklov parameter tend to zero, next let
		the terminal cutoff converge to the time indicator, and finally let
		$\varepsilon\downarrow0$.  The weak lower semicontinuity of the gradient
		term and monotone convergence for the nonnegative zeroth-order terms give
		the inequality below.  Since $b\geq1$, we suppress this notation.  Fix
		$t\in J$ and use
		\[
		\psi=\lambda\eta^2f^b
		\]
		in \eqref{eq:weak-form}, integrating over $M\times(t_0,t]$.
		
		For the time derivative, the assumption $\lambda(t_0)=0$ gives
		\begin{align}
			\int_{t_0}^t\!\!\int\lambda\eta^2f^bf_t
			&=\frac1{b+1}\int_{t_0}^t\!\!\int
			\lambda\eta^2\partial_t(f^{b+1})\notag\\
			&=\frac1{b+1}\int\lambda(t)\eta^2f^{b+1}(t)\dd\mu
			-\frac1{b+1}\int_{t_0}^t\!\!\int
			\lambda'\eta^2f^{b+1}.
			\label{eq:Cacc-time}
		\end{align}
		The spatial derivative of the test function is
		\[
		\nabla(\lambda\eta^2f^b)
		=2\lambda\eta f^b\nabla\eta
		+b\lambda\eta^2f^{b-1}\nabla f.
		\]
		Therefore
		\begin{align}
			\int_{t_0}^t\!\!\int h
			\langle\nabla(\lambda\eta^2f^b),\nabla f\rangle
			={}b\int_{t_0}^t\!\!\int
			h\lambda\eta^2f^{b-1}|\nabla f|^2
			+2\int_{t_0}^t\!\!\int
			h\lambda\eta f^b\langle\nabla\eta,\nabla f\rangle.
			\label{eq:Cacc-diffusion}
		\end{align}
		
		We next estimate the drift term.  By the definition of $\Lambda$ and the
		identity $y=|\nabla v|^2$,
		\begin{align*}
			|(2-h'(v))\lambda\eta^2f^b
			\langle\nabla v,\nabla f\rangle|
			&\leq\Lambda\lambda\eta^2f^b\sqrt y\,|\nabla f|\\
			&=\Lambda
			\bigl(\sqrt{h\lambda}\,\eta f^{(b-1)/2}|\nabla f|\bigr)
			\bigl(h^{-1/2}\sqrt\lambda\,\eta
			f^{(b+1)/2}\sqrt y\bigr).
		\end{align*}
		Young's inequality gives
		\begin{align}
			|(2-h'(v))\lambda\eta^2f^b
			\langle\nabla v,\nabla f\rangle|
			\leq{}&\frac b4h\lambda\eta^2f^{b-1}|\nabla f|^2
			+\frac{\Lambda^2}{bh}\lambda\eta^2yf^{b+1}\notag\\
			\leq{}&\frac b4h\lambda\eta^2f^{b-1}|\nabla f|^2
			+\frac{d_0}{2}\lambda\eta^2yf^{b+1},
			\label{eq:Cacc-drift}
		\end{align}
		where the last inequality follows from \eqref{eq:b-threshold}.  Similarly,
		the cutoff cross term in \eqref{eq:Cacc-diffusion} satisfies
		\begin{align}
			2h\lambda\eta f^b|\nabla\eta|\,|\nabla f|
			\leq{}&\frac b4h\lambda\eta^2f^{b-1}|\nabla f|^2
			+\frac{4h}{b}\lambda|\nabla\eta|^2f^{b+1}\notag\\
			\leq{}&\frac b4h\lambda\eta^2f^{b-1}|\nabla f|^2
			+\frac{4h_+}{b}\lambda|\nabla\eta|^2f^{b+1}.
			\label{eq:Cacc-cross}
		\end{align}
		
		Substituting \eqref{eq:Cacc-time}--\eqref{eq:Cacc-cross} into the weak
		formulation, we move the favorable $f^{b+2}$- and $yf^{b+1}$-terms to the
		left.  After absorbing the two $b/4$ contributions to the gradient term, we find
		\begin{align*}
			&\frac1{b+1}\int\lambda(t)\eta^2f^{b+1}(t)\dd\mu
			+\frac b2\int_{t_0}^t\!\!\int
			h\lambda\eta^2f^{b-1}|\nabla f|^2\\
			&\quad+c_0\int_{t_0}^t\!\!\int\lambda\eta^2f^{b+2}
			+\frac{d_0}{2}\int_{t_0}^t\!\!\int\lambda\eta^2yf^{b+1}\\
			&\leq\int_{t_0}^t\!\!\int
			\left(\frac{|\lambda'|}{b+1}\eta^2
			+\frac{4h_+}{b}\lambda|\nabla\eta|^2\right)f^{b+1}.
		\end{align*}
		For each $t\in J$, the right-hand side is bounded by its integral over all of $J$.
		Taking the supremum controls the boundary term, while choosing $t=t_1$ controls
		the full space--time integrals.  Adding these two bounds introduces the factor
		$2$ on the right and proves \eqref{eq:Caccioppoli}.
	\end{proof}
	
	\begin{lemma}\label{lem:master}
		Under the assumptions of Lemma~\ref{lem:Caccioppoli}, suppose in addition
		that $n>2$ and $\Ric\geq-(n-1)Kg$ on $B_{2R}$. Then there are constants
		$\kappa_1,\kappa_2,\kappa_3>0$, depending only on
		$n,h_-,h_+$, and $c_0$, such that, whenever $b\geq b_*$,
		\begin{align}
			&\left(\int_J\!\!\int_{B_{2R}}
			\lambda^\chi\eta^{2\chi}f^{\chi(b+1)}\dd\mu\dd t\right)^{1/\chi}
			\notag\\
			&\quad\leq
			\mathcal S_R^{1/\chi}\kappa_1
			\int_J\!\!\int_{B_{2R}}
			\Bigl[|\lambda'|\eta^2
			+\kappa_2\lambda\bigl(|\nabla\eta|^2+R^{-2}\eta^2\bigr)\Bigr]
			f^{b+1}\dd\mu\dd t
			-\mathcal S_R^{1/\chi}\kappa_1\kappa_3b
			\int_J\!\!\int_{B_{2R}}\lambda\eta^2f^{b+2}\dd\mu\dd t.
			\label{eq:master}
		\end{align}
		The full expression on the right-hand side is nonnegative.
	\end{lemma}
	
	\begin{proof}
		Set
		\begin{equation}\label{eq:omega-energy}
			\omega=\lambda^{1/2}\eta f^{(b+1)/2}.
		\end{equation}
		Then $\omega^2=\lambda\eta^2f^{b+1}$ and
		$\omega^{2\chi}=\lambda^\chi\eta^{2\chi}f^{\chi(b+1)}$.  Since $\lambda$ depends
		only on time,
		\begin{align*}
			\nabla\omega
			&=\lambda^{1/2}f^{(b+1)/2}\nabla\eta
			+\frac{b+1}{2}\lambda^{1/2}\eta f^{(b-1)/2}\nabla f,
		\end{align*}
		and hence
		\begin{equation}\label{eq:grad-omega}
			|\nabla\omega|^2
			\leq2\lambda|\nabla\eta|^2f^{b+1}
			+\frac{(b+1)^2}{2}\lambda\eta^2f^{b-1}|\nabla f|^2.
		\end{equation}
		Introduce
		\begin{align}
			\mathcal R_b
			&=2\int_J\!\!\int
			\left(\frac{|\lambda'|}{b+1}\eta^2
			+\frac{4h_+}{b}\lambda|\nabla\eta|^2\right)f^{b+1},
			\label{eq:Rb}\\
			\mathcal I_b
			&=\int_J\!\!\int\lambda\eta^2f^{b+2}.
			\label{eq:Ib}
		\end{align}
		Dropping only nonnegative terms from \eqref{eq:Caccioppoli} gives
		\begin{align}
			\sup_{t\in J}\int \omega^2\dd\mu
			&\leq(b+1)(\mathcal R_b-c_0\mathcal I_b),
			\label{eq:sup-h-bound}\\
			\int_J\!\!\int\lambda\eta^2f^{b-1}|\nabla f|^2
			&\leq\frac{2}{b h_-}(\mathcal R_b-c_0\mathcal I_b).
			\label{eq:grad-f-bound}
		\end{align}
		In particular, $\mathcal R_b-c_0\mathcal I_b\geq0$.
		
		By \eqref{eq:grad-omega} and \eqref{eq:grad-f-bound},
		\begin{align}
			\int_J\!\!\int\bigl(|\nabla\omega|^2+R^{-2}\omega^2\bigr)
			\leq{}2\int_J\!\!\int\lambda|\nabla\eta|^2f^{b+1}
			+R^{-2}\int_J\!\!\int\lambda\eta^2f^{b+1}
			+\frac{(b+1)^2}{b h_-}(\mathcal R_b-c_0\mathcal I_b).
			\label{eq:space-energy-bound}
		\end{align}
		Because $b\geq1$, one has $b+1\leq2b$ and
		$(b+1)^2/b\leq4b$.  Fix a dimension-dependent constant $C_0=C_0(n)\geq1$
		large enough to dominate these numerical coefficients.  Then both the
		quantity in \eqref{eq:sup-h-bound} and the right-hand side of
		\eqref{eq:space-energy-bound} are bounded by
		\begin{align}
			\mathcal E_b
			={}C_0(1+h_-^{-1})b(\mathcal R_b-c_0\mathcal I_b)
			+2\int_J\!\!\int\lambda|\nabla\eta|^2f^{b+1}
			+R^{-2}\int_J\!\!\int\lambda\eta^2f^{b+1}.
			\label{eq:Eb-definition}
		\end{align}
		Expanding $\mathcal R_b$ in \eqref{eq:Eb-definition} shows explicitly
		that
		\begin{align*}
			\mathcal E_b
			\leq{}&
			2C_0(1+h_-^{-1})\int_J\!\!\int|\lambda'|\eta^2f^{b+1}
			+\bigl(8C_0(1+h_-^{-1})h_++2\bigr)
			\int_J\!\!\int\lambda|\nabla\eta|^2f^{b+1}\\
			&+R^{-2}\int_J\!\!\int\lambda\eta^2f^{b+1}
			-C_0(1+h_-^{-1})c_0b\,\mathcal I_b.
		\end{align*}
		Thus, for example, one may choose
		\begin{equation}\label{eq:kappa-choice}
			\kappa_1=2C_0(1+h_-^{-1}),\qquad
			\kappa_2=2+4h_+,\qquad
			\kappa_3=\frac{c_0}{4}.
		\end{equation}
		With this choice, every positive coefficient on the right is large
		enough, whereas
		$\kappa_1\kappa_3\leq C_0(1+h_-^{-1})c_0$; hence the negative term is
		no more negative than the one already present in $\mathcal E_b$.
		Consequently,
		\begin{align}
			\mathcal E_b
			\leq{}&\kappa_1\int_J\!\!\int
			\Bigl[|\lambda'|\eta^2
			+\kappa_2\lambda\bigl(|\nabla\eta|^2+R^{-2}\eta^2\bigr)\Bigr]f^{b+1}
			-\kappa_1\kappa_3b\int_J\!\!\int\lambda\eta^2f^{b+2}.
			\label{eq:Eb-structural}
		\end{align}
		The full right-hand side of \eqref{eq:Eb-structural} is nonnegative
		because it bounds $\mathcal E_b\geq0$.
		
		Apply the parabolic Sobolev inequality \eqref{eq:parabolic-Sobolev} to $\omega$,
		using the same approximation procedure if necessary.
		The two energy factors there are both bounded by $\mathcal E_b$, and
		\[
		\frac{2}{n\chi}+\frac1\chi=1.
		\]
		Therefore
		\begin{align*}
			\left(\int_J\!\!\int \omega^{2\chi}\right)^{1/\chi}
			&\leq\mathcal S_R^{1/\chi}
			\mathcal E_b^{2/(n\chi)}\mathcal E_b^{1/\chi}
			=\mathcal S_R^{1/\chi}\mathcal E_b.
		\end{align*}
		Substituting \eqref{eq:Eb-structural} and the expression for $\omega^{2\chi}$
		proves \eqref{eq:master}.
	\end{proof}
	
	\subsection{Initial higher integrability}
	
	\begin{lemma}\label{lem:initial-Lp}
		Let $0<\sigma<\sigma'\leq1$ and $R\leq\rho\leq3R/2$.  Choose
		$\beta_0=\beta_0(\mathscr D)>0$ sufficiently large, set
		\begin{equation}\label{eq:iteration-exponents-initial}
			b_0:=\beta_0(1+\sqrt K R),\qquad
			\chi:=1+\frac2n,\qquad b_1:=\chi(b_0+1),
		\end{equation}
		and assume $b_0\geq b_*$.  Then
		\begin{align}
			\|f\|_{L^{b_1}(I_\sigma\times B_\rho)}
			\leq{}&
			\frac{C_{\mathscr D}}{(\sigma'-\sigma)^3}
			\frac{1+\sqrt K R}{R^2}
			\bigl(R^2\vol(B_{2R})\bigr)^{1/b_1}.
			\label{eq:initial-Lp}
		\end{align}
	\end{lemma}
	
	\begin{proof}
		Put
		\begin{equation}\label{eq:p-definition}
			p=b_0+1,
			\qquad b_1=\chi p.
		\end{equation}
		The choice of $\beta_0$ ensures that $b_0=p-1\geq b_*$.  In
		Lemma~\ref{lem:master}, take $J=(s-\sigma'R^2,s]$.  Choose a smooth
		time cutoff $\widetilde\lambda$ satisfying
		\[
		0\leq\widetilde\lambda\leq1,
		\qquad \widetilde\lambda=0
		\quad\hbox{on }(-\infty,s-\sigma'R^2],
		\qquad \widetilde\lambda=1\quad\hbox{on }I_\sigma,
		\]
		and
		\begin{equation}\label{eq:initial-time-cutoff}
			|\widetilde\lambda'|
			\leq\frac{C}{(\sigma'-\sigma)R^2}.
		\end{equation}
		Choose $\widetilde\eta\in C_0^\infty(B_{2R})$ with
		$0\leq\widetilde\eta\leq1$, $\widetilde\eta=1$ on $B_\rho$, and
		$|\nabla\widetilde\eta|\leq C/R$.  The cutoffs may be chosen flat at
		the boundary of their zero sets, so that the positive real powers below
		remain smooth.  In \eqref{eq:master} take
		\begin{equation}\label{eq:powered-cutoffs}
			\lambda=\widetilde\lambda^{p+1},
			\qquad \eta=\widetilde\eta^{p+1},
			\qquad W=\widetilde\lambda\widetilde\eta^2,
			\qquad b=b_0=p-1.
		\end{equation}
		
		The purpose of the powers in \eqref{eq:powered-cutoffs} is that all terms on
		the right of \eqref{eq:master} acquire the same basic weight.  Indeed,
		\begin{align*}
			|\lambda'|\eta^2f^p
			&=(p+1)\widetilde\lambda^p|\widetilde\lambda'|
			\widetilde\eta^{2p+2}f^p
			\leq(p+1)|\widetilde\lambda'|(Wf)^p,\\
			\lambda|\nabla\eta|^2f^p
			&=(p+1)^2\widetilde\lambda^{p+1}\widetilde\eta^{2p}
			|\nabla\widetilde\eta|^2f^p
			\leq(p+1)^2|\nabla\widetilde\eta|^2(Wf)^p,\\
			R^{-2}\lambda\eta^2f^p
			&\leq R^{-2}(Wf)^p,
		\end{align*}
		while the favorable term is exactly
		\[
		\lambda\eta^2f^{p+1}=(Wf)^{p+1}.
		\]
		Since $0<\sigma'-\sigma<1$, all positive cutoff terms are bounded by
		$L_0(Wf)^p$, where
		\begin{equation}\label{eq:L0}
			L_0=\frac{C(p+1)^2(1+\kappa_2)}{(\sigma'-\sigma)^2R^2}.
		\end{equation}
		Thus the integrand on the right-hand side of \eqref{eq:master} is bounded by
		\[
		L_0(Wf)^p-\kappa_3b_0(Wf)^{p+1}.
		\]
		For every $x\geq0$,
		\begin{equation}\label{eq:young-polynomial}
			L_0x^p-\kappa_3b_0x^{p+1}
			\leq L_0^{p+1}(\kappa_3b_0)^{-p}.
		\end{equation}
		For completeness, the maximum of the left-hand side occurs at
		$x=pL_0/[(p+1)\kappa_3b_0]$, and its value is
		\[
		\frac{p^p}{(p+1)^{p+1}}L_0^{p+1}(\kappa_3b_0)^{-p},
		\]
		which is bounded by the right-hand side of \eqref{eq:young-polynomial}.
		
		The support of $W$ is contained in a time interval of length at most $R^2$
		and in $B_{2R}$.  Hence
		\[
		|\supp W|\leq R^2\vol(B_{2R}).
		\]
		Applying \eqref{eq:master}, using \eqref{eq:young-polynomial}, and recalling
		that the cutoffs equal one on $I_\sigma\times B_\rho$, we find
		\begin{align*}
			\left(\int_{I_\sigma\times B_\rho}f^{\chi p}\right)^{1/\chi}
			\leq{}&\mathcal S_R^{1/\chi}\kappa_1
			L_0^{p+1}(\kappa_3b_0)^{-p}
			R^2\vol(B_{2R}).
		\end{align*}
		Taking the $p$-th root gives
		\begin{align}
			\|f\|_{L^{b_1}(I_\sigma\times B_\rho)}
			\leq{}&\mathcal S_R^{1/b_1}\kappa_1^{1/p}
			L_0^{1+1/p}(\kappa_3b_0)^{-1}
			\bigl(R^2\vol(B_{2R})\bigr)^{1/p}.
			\label{eq:initial-before-simplification}
		\end{align}
		
		We finally simplify the factors in \eqref{eq:initial-before-simplification}.
		First, since $p\geq2$,
		\[
		(\sigma'-\sigma)^{-2(1+1/p)}
		\leq(\sigma'-\sigma)^{-3}.
		\]
		Next, from \eqref{eq:S_R} and $b_1=\chi p$,
		\begin{align*}
			\mathcal S_R^{1/b_1}
			R^{-2-2/p}\bigl(R^2\vol(B_{2R})\bigr)^{1/p}
			\leq C_nR^{-2}
			\bigl(R^2\vol(B_{2R})\bigr)^{1/b_1}.
		\end{align*}
		Here the exponential factor is uniformly bounded because
		$b_1=\chi[\beta_0(1+\sqrt K R)+1]$.  Finally,
		\begin{equation}\label{eq:p-factor-bound}
			\frac{(p+1)^{2+2/p}}{b_0}
			\leq C_{\mathscr D}(1+\sqrt K R).
		\end{equation}
		The constant $C_{\mathscr D}$ is essential here because $p$ contains the
		structural parameter $\beta_0$.  Since the remaining factors
		$\kappa_1,\kappa_2,\kappa_3$ also depend only on $\mathscr D$, they may be
		absorbed into $C_{\mathscr D}$.  Substitution into
		\eqref{eq:initial-before-simplification} yields \eqref{eq:initial-Lp}.
	\end{proof}
	
	\subsection{Nash--Moser iteration}
	
	\begin{lemma}\label{lem:one-step}
		Fix $0<\tau<\tau'<1$ and put $\delta=\tau'-\tau$.  Define
		\begin{equation}\label{eq:bk}
			b_{k+1}=\chi b_k,
			\qquad b_1=\chi(b_0+1),
		\end{equation}
		and
		\begin{equation}\label{eq:tau-R-k}
			\tau_k=\tau+\frac{\delta}{2^{k-1}},
			\qquad R_k=R+\frac{\delta R}{2^k},
			\qquad Q_k=I_{\tau_k}\times B_{R_k}.
		\end{equation}
		Then
		\begin{equation}\label{eq:one-step}
			\|f\|_{L^{b_{k+1}}(Q_{k+1})}
			\leq
			\mathcal S_R^{1/b_{k+1}}
			\left(\frac{D_{\mathscr D}4^k}{\delta^2R^2}\right)^{1/b_k}
			\|f\|_{L^{b_k}(Q_k)},
		\end{equation}
		where
		\begin{equation}\label{eq:DF}
			D_{\mathscr D}=C_n\kappa_1(1+\kappa_2).
		\end{equation}
	\end{lemma}
	
	\begin{proof}
		The cylinders are nested and decrease to the desired cylinder:
		$Q_{k+1}\Subset Q_k$ and
		$\bigcap_{k\geq1}Q_k=I_\tau\times B_R$.  Choose a time cutoff $\lambda_k$
		such that
		\[
		0\leq\lambda_k\leq1,
		\qquad \lambda_k=0\quad\hbox{before }s-\tau_kR^2,
		\qquad \lambda_k=1\quad\hbox{on }I_{\tau_{k+1}},
		\]
		and
		\begin{equation}\label{eq:lambda-k-bound}
			|\lambda_k'|
			\leq\frac{C2^k}{\delta R^2}.
		\end{equation}
		Choose $\eta_k\in C_0^\infty(B_{R_k})$ satisfying
		$0\leq\eta_k\leq1$, $\eta_k=1$ on $B_{R_{k+1}}$, and
		\begin{equation}\label{eq:eta-k-bound}
			|\nabla\eta_k|^2
			\leq\frac{C4^k}{\delta^2R^2}.
		\end{equation}
		
		Apply \eqref{eq:master} on $J=I_{\tau_k}$ with $b+1=b_k$, and discard the favorable negative
		term.  Because $b_k-1\geq b_*$ and $0<\delta<1$, the cutoff estimates imply
		\begin{align*}
			|\lambda_k'|\eta_k^2
			+\kappa_2\lambda_k\bigl(|\nabla\eta_k|^2+R^{-2}\eta_k^2\bigr)
			\leq C_n(1+\kappa_2)\frac{4^k}{\delta^2R^2}.
		\end{align*}
		Since $\lambda_k=\eta_k=1$ on $Q_{k+1}$ and both cutoffs are supported in
		$Q_k$, \eqref{eq:master} gives
		\begin{align*}
			\left(\int_{Q_{k+1}}f^{\chi b_k}\right)^{1/\chi}
			\leq\mathcal S_R^{1/\chi}
			\frac{D_{\mathscr D}4^k}{\delta^2R^2}\int_{Q_k}f^{b_k}.
		\end{align*}
		Taking the $b_k$-th root and using $b_{k+1}=\chi b_k$ proves
		\eqref{eq:one-step}.
	\end{proof}
	
	\begin{lemma}\label{lem:NashMoser-limit}
		Under the assumptions of Lemma~\ref{lem:one-step},
		\begin{align}
			\|f\|_{L^\infty(I_\tau\times B_R)}
			\leq{}&C_{\mathscr D}
			\delta^{-(n+2)/b_1}
			\bigl(R^2\vol(B_{2R})\bigr)^{-1/b_1}
			\|f\|_{L^{b_1}(Q_1)}.
			\label{eq:NashMoser-limit}
		\end{align}
	\end{lemma}
	
	\begin{proof}
		Iterating \eqref{eq:one-step} from $k=1$ to $k=N$ gives
		\begin{align}
			\|f\|_{L^{b_{N+1}}(Q_{N+1})}
			\leq{}&
			\mathcal S_R^{\sum_{k=1}^N b_{k+1}^{-1}}
			\left(\frac{D_{\mathscr D}}{\delta^2R^2}\right)^{\sum_{k=1}^N b_k^{-1}}
			4^{\sum_{k=1}^N k/b_k}
			\|f\|_{L^{b_1}(Q_1)}.
			\label{eq:finite-iteration}
		\end{align}
		Since $b_k=\chi^{k-1}b_1$ and $\chi=(n+2)/n$, the relevant geometric
		series are
		\begin{align}
			\sum_{k=1}^\infty\frac1{b_k}
			&=\frac1{b_1}\frac1{1-\chi^{-1}}
			=\frac{n+2}{2b_1}
			=\frac{n}{2(b_0+1)},\label{eq:series1}\\
			\sum_{k=1}^\infty\frac1{b_{k+1}}
			&=\frac1{\chi b_1}\frac1{1-\chi^{-1}}
			=\frac{n}{2b_1},\label{eq:series2}\\
			\sum_{k=1}^\infty\frac{k}{b_k}
			&=\frac1{b_1}\frac1{(1-\chi^{-1})^2}
			=\frac{(n+2)^2}{4b_1}
			=\frac{n(n+2)}{4(b_0+1)}.\label{eq:series3}
		\end{align}
		Because $b_{N+1}\to\infty$ and $Q_{N+1}\downarrow I_\tau\times B_R$,
		the left-hand side of \eqref{eq:finite-iteration} converges to the essential
		supremum of $f$ on $I_\tau\times B_R$.  Letting $N\to\infty$ gives
		\begin{align}
			\|f\|_{L^\infty(I_\tau\times B_R)}
			\leq{}\mathcal S_R^{n/(2b_1)}
			D_{\mathscr D}^{n/[2(b_0+1)]}
			\delta^{-(n+2)/b_1}R^{-n/(b_0+1)}
			\times4^{n(n+2)/[4(b_0+1)]}
			\|f\|_{L^{b_1}(Q_1)}.
			\label{eq:limit-before-scaling}
		\end{align}
		Using \eqref{eq:S_R}, the exponential term in
		$\mathcal S_R^{n/(2b_1)}$ is uniformly bounded, since
		$b_1=\chi[\beta_0(1+\sqrt K R)+1]$.  The factor involving $4$ is uniformly bounded and is also
		absorbed into $C_n$.  The remaining scale and volume factors satisfy
		\begin{align*}
			R^{n/b_1-n/(b_0+1)}\vol(B_{2R})^{-1/b_1}
			=R^{-2/b_1}\vol(B_{2R})^{-1/b_1}
			=\bigl(R^2\vol(B_{2R})\bigr)^{-1/b_1}.
		\end{align*}
		Finally, $D_{\mathscr D}=C_n\kappa_1(1+\kappa_2)$ and all the remaining fixed
		factors depend only on $\mathscr D$.  Absorbing them into
		$C_{\mathscr D}$ gives \eqref{eq:NashMoser-limit}.
	\end{proof}

	\begin{proposition}\label{prop:abstract-NashMoser}
		Suppose that $n>2$ and $\Ric\geq-(n-1)Kg$ on $B_{2R}$. Let $f\geq0$
		be locally Lipschitz and satisfy, in distributions on
		$\Q_{2R}$,
		\[
		\cL f\le2\langle\nabla v,\nabla f\rangle
		-c_0f^2-d_0yf,
		\qquad y=|\nabla v|^2,
		\]
		where $c_0,d_0>0$, $0<h_-\leq h\leq h_+<\infty$, and
		$\Lambda:=\sup_{\Q_{2R}}|2-h'(v)|<\infty$.  Set
		$\mathscr D=(n,h_-,h_+,c_0,d_0,\Lambda)$.  Then, for every
		$0<\tau<1$,
		\begin{equation}\label{eq:abstract-NashMoser-estimate}
			\|f\|_{L^\infty(I_\tau\times B_R)}
			\le
			\frac{C(n,h_-,h_+,c_0,d_0,\Lambda)}{(1-\tau)^{3+n/2}}
			\frac{1+\sqrt K R}{R^2}.
		\end{equation}
	\end{proposition}
	
	\begin{proof}
		Choose $\beta_0=\beta_0(\mathscr D)\ge1$ so large that
		\[
		\beta_0\ge \max\left\{1,\frac{2\Lambda^2}{d_0h_-}\right\},
		\]
		and set
		\begin{equation}\label{eq:iteration-exponents-main-proof}
			b_0=\beta_0(1+\sqrt K R),
			\qquad \chi=1+\frac2n,
			\qquad b_1=\chi(b_0+1).
		\end{equation}
		Then $b_0\ge b_*$, so all preceding energy and iteration estimates apply.
		
		Fix $0<\tau<1$ and put
		\[
		\tau'=\frac{1+\tau}{2},
		\qquad \delta=\tau'-\tau=\frac{1-\tau}{2}.
		\]
		The first cylinder in the infinite iteration is
		$Q_1=I_{\tau'}\times B_{R_1}$ with $R\le R_1\le3R/2$.
		Lemma~\ref{lem:initial-Lp}, used with $\sigma=\tau'$, $\sigma'=1$, and
		$\rho=R_1$, gives
		\[
		\|f\|_{L^{b_1}(Q_1)}
		\le
		\frac{C_{\mathscr D}}{(1-\tau')^3}
		\frac{1+\sqrt K R}{R^2}
		\bigl(R^2\vol(B_{2R})\bigr)^{1/b_1}.
		\]
		Inserting this estimate into Lemma~\ref{lem:NashMoser-limit} cancels the
		volume factors and yields
		\begin{equation}\label{eq:sharp-in-proof}
			\|f\|_{L^\infty(I_\tau\times B_R)}
			\le
			\frac{C_{\mathscr D}}{(1-\tau)^{3+(n+2)/b_1}}
			\frac{1+\sqrt K R}{R^2}.
		\end{equation}
		Since
		\[
		\frac{n+2}{b_1}=\frac{n}{b_0+1}\le\frac n2
		\]
		because $b_0\ge1$, \eqref{eq:abstract-NashMoser-estimate} follows.
	\end{proof}
	
	\begin{proof}[\bf Proof of Theorem~\ref{thm:barrier-main}]
		The uniform parabolicity condition in
		\eqref{eq:structural-assumptions-main} gives the bounds $h_-$ and $h_+$,
		while \eqref{eq:hprime-upper-energy} gives a finite drift bound depending
		only on $\delta_-$ and $\delta_+$.
		Under the lower bound \eqref{eq:delta-minus}, the constants
		\begin{equation}\label{eq:c0-d0-main-proof}
			c_0=\frac{2+n\alpha\delta_-}{n\alpha^2h_+},
			\qquad
			d_0=\frac{4(\alpha-1)+n\alpha^2\delta_-}
			{n\alpha^2h_+}
		\end{equation}
		are positive.  The condition \eqref{eq:barrier-Riccati-main} is exactly
		\eqref{eq:new-barrier-condition}.  Hence Lemma~\ref{lem:pointwise} shows that
		$f=(H_\alpha-\phi)_+$ is locally Lipschitz and satisfies the hypothesis of
		Proposition~\ref{prop:abstract-NashMoser}.  Consequently,
		\begin{equation}\label{eq:local-Halpha-estimate-proof}
			\sup_{B_R(x_0)\times I_\tau}(H_\alpha-\phi)_+
			\le
			\frac{C(n,h_-,h_+,c_0,d_0,\delta_-,\delta_+)}
			{(1-\tau)^{3+n/2}}
			\frac{1+\sqrt K R}{R^2}.
		\end{equation}
		This is the asserted local differential estimate.
	\end{proof}

	\begin{remark}[The parabolic time-layer exponent]\label{rem:time-exponent}
		The initial higher-integrability estimate contributes the factor
		$(1-\tau')^{-3}$.  At the $k$-th exponent-improvement step, the time and
		space cutoffs together contribute $\delta^{-2/b_k}$.  Consequently, the
		infinite iteration contributes
		\[
		\delta^{-2\sum_{k\geq1}b_k^{-1}}
		=\delta^{-(n+2)/b_1}.
		\]
		Thus the numerator $n+2$ is the parabolic dimension; replacing it by $n$
		at this stage would omit the time-cutoff contribution.
	\end{remark}
	
	\begin{remark}[Dimension two]\label{rem:dimension-two}
		The proof above is written for $n>2$ because it uses the critical local
		Sobolev exponent $2n/(n-2)$.  A two-dimensional variant would require a
		separate local Sobolev inequality with a chosen finite exponent (or an
		effective Sobolev dimension $N>2$) and a corresponding repetition of the
		iteration.  The Bochner calculation itself would still use the true
		geometric dimension $n=2$.  We do not pursue that variant here.
	\end{remark}

	\subsection{Proofs of the corollaries}\label{sec:consequences}
	
	For the choice \eqref{eq:explicit-barrier},
	\[
	2(n-1)Kh_+-d_0\phi_K
	=2\sqrt{c_0e_0}\,\phi_K.
	\]
	Thus the two quadratic terms in \eqref{eq:new-barrier-condition} are equal,
	and $\phi_K$ is an admissible constant barrier.
	
	\begin{proof}[\bf Proof of Theorem~\ref{thm:Harnack-additive}]
		The differential estimate \eqref{eq:local-Halpha-estimate-proof} gives
		\[
		v_t\ge\frac1\alpha|\nabla v|^2
		-\frac1\alpha(\phi_K+E_{R,\tau}).
		\]
		Let $\gamma:[t_1,t_2]\to B_R(x_0)$ be a constant-speed minimizing
		geodesic from $x_1$ to $x_2$.  Along $\gamma$,
		\begin{align*}
			\frac{\dd}{\dd t}v(\gamma(t),t)
			\ge\frac1\alpha|\nabla v|^2-|\nabla v|\,|\dot\gamma|
			-\frac1\alpha(\phi_K+E_{R,\tau})
			\ge-\frac\alpha4|\dot\gamma|^2
			-\frac1\alpha(\phi_K+E_{R,\tau}).
		\end{align*}
		Integration gives \eqref{eq:additive-Harnack}.
	\end{proof}
	
	\begin{proof}[\bf Proof of Theorem~\ref{thm:global}]
		Fix $(x,t)\in M\times(-\infty,T)$, take $s=t$ and $\tau=1/2$, and
		note that the assumed smooth pressure satisfies the pressure equation
		pointwise. We may therefore apply \eqref{eq:local-Halpha-estimate-proof}
		with $\phi=\phi_K$ on
		$B_{2R}(x)\times(t-R^2,t+\varepsilon)$, where $t+\varepsilon<T$.
		The constants are uniform in $R$, so
		\[
		H_\alpha(x,t)-\phi_K
		\le C\frac{1+\sqrt K R}{R^2}.
		\]
		Letting $R\to\infty$ proves \eqref{eq:global-Harnack-estimate}.
	\end{proof}
	
	\begin{proof}[\bf Proof of Theorem~\ref{thm:Liouville}]
		Theorem~\ref{thm:global} with $K=0$ gives $H_\alpha\le0$.  On the
		other hand, $v_t=(F'(u)/u)u_t\le0$ and $h=F'(u)>0$, so
		$z=v_t\le0$.  Consequently
		\[
		0\ge H_\alpha=y-\alpha z\ge y\ge0.
		\]
		Thus $y=0$ and $z=0$, whence $\nabla v=0$ and $v_t=0$.  Since the
		pressure transformation is strictly increasing and $M$ is connected, $u$
		is constant on space--time.
	\end{proof}

	\section{Proof of Frequency Monotonicity}\label{sec:frequency}
	
	This section proves Theorem~\ref{thm:frequency-main}. The proof uses only
	the energy structure of the filtration equation. In particular, no
	curvature assumption is required.
	
	Let $\mathcal I\subset\mathbb R$ be an interval, let
	$F\in C^2(\mathcal I)$ satisfy $F'>0$, and fix
	$a\in\overline{\mathcal I}$ such that $F$ extends continuously to $a$.
	Define
	\begin{equation}\label{eq:frequency-G-Psi}
		G_a(s)=F(s)-F(a),
		\qquad
		\Psi_a(s)=\int_a^sG_a(r)\,\dd r.
	\end{equation}
	Since $F$ is strictly increasing, $\Psi_a(s)\geq0$ for every
	$s\in\mathcal I$. Indeed, this is immediate when $s\geq a$, while for
	$s<a$,
	\[
	\Psi_a(s)
	=-\int_s^a\bigl(F(r)-F(a)\bigr)\,\dd r\geq0.
	\]
	Moreover, $\Psi_a(s)=0$ only when $s=a$.
	
	Let $u:M\times(0,T)\to\mathcal I$ be a smooth solution of
	\[
	u_t=\Delta F(u),
	\]
	and write
	\[
	w=F(u).
	\]
	We introduce the entropy, its dissipation, and the second dissipation by
	\begin{equation}\label{eq:frequency-E-D-A}
		\mathcal E_a(t)
		=\int_M\Psi_a(u)\,\dd\mu,
		\qquad
		\mathcal D(t)
		=\int_M|\nabla w|^2\,\dd\mu,
		\qquad
		\mathcal A(t)
		=\int_MF'(u)(\Delta w)^2\,\dd\mu.
	\end{equation}
	
	\begin{definition}[Frequency-admissible solution]
		\label{def:frequency-admissible}
		A smooth solution $u$ is called frequency-admissible relative to $a$ if,
		for every compact interval $I'\Subset(0,T)$,
		\begin{equation}\label{eq:frequency-admissibility-uniform}
			\Psi_a(u)\in C\bigl(I';L^1(M)\bigr),
			\qquad
			\nabla w\in C\bigl(I';L^2(M)\bigr),
		\end{equation}
		and
		\begin{equation}\label{eq:frequency-admissibility-L2}
			\frac{G_a(u)}{\sqrt{F'(u)}},\quad
			\sqrt{F'(u)}\,\nabla w,\quad
			\sqrt{F'(u)}\,\Delta w
			\in L^2(M\times I').
		\end{equation}
		These conditions are automatic for smooth solutions on a closed
		manifold. On a complete noncompact manifold, they ensure that the
		cutoff errors below vanish and that the energy identities hold in the
		locally absolutely continuous sense.
	\end{definition}
	
	\begin{proposition}[An admissibility criterion]
		\label{prop:frequency-admissible-sufficient}
		A smooth solution is frequency-admissible on every compact time interval
		if $M$ is closed. The same conclusion holds on a complete noncompact
		manifold if, for every $I'\Subset(0,T)$, the range of $u$ is contained in
		a compact interval $K\Subset\mathcal I$ on which
		\[
		0<\inf_KF'\leq\sup_KF'<\infty,
		\qquad
		\sup_K\frac{G_a^2}{F'\Psi_a}<\infty,
		\]
		where the ratio is understood by continuous extension at $a$, and
		\[
		\Psi_a(u)\in C(I';L^1(M)),\qquad
		\nabla w\in C(I';L^2(M)),\qquad
		\Delta w\in L^2(M\times I').
		\]
	\end{proposition}
	
	\begin{proof}
		On a closed manifold the assertion follows from smoothness and compactness.
		In the noncompact case, boundedness of $F'$ on $K$ gives the two weighted
		$L^2$ conditions involving $\nabla w$ and $\Delta w$. If the displayed
		ratio is bounded by $\kappa_K$, then
		\[
		\frac{G_a(u)^2}{F'(u)}\leq\kappa_K\Psi_a(u),
		\]
		so the remaining $L^2$ condition follows from the finite entropy. These are
		exactly the requirements in Definition~\ref{def:frequency-admissible}.
	\end{proof}
	
	\begin{proposition}[Evolution identities]
		\label{prop:frequency-evolution}
		For every frequency-admissible solution, $\mathcal E_a$ and $\mathcal D$
		are locally absolutely continuous and
		\begin{equation}\label{eq:frequency-evolution}
			\mathcal E_a'(t)=-\mathcal D(t),
			\qquad
			\mathcal D'(t)=-2\mathcal A(t)
		\end{equation}
		for almost every $t\in(0,T)$.
		Consequently, $\mathcal E_a\in W^{2,1}_{\mathrm{loc}}(0,T)$ and
		\begin{equation}\label{eq:frequency-E-second}
			\mathcal E_a''(t)=2\mathcal A(t)\geq0
		\end{equation}
		for almost every $t\in(0,T)$; equivalently, the same identity holds
		in the sense of distributions. In particular, $\mathcal E_a$ and
		$\mathcal D$ are nonincreasing, whereas $\mathcal E_a$ is convex.
	\end{proposition}
	
	\begin{proof}
		We first give the calculation formally. Since
		$\Psi_a'(u)=G_a(u)$ and $u_t=\Delta w$, we have
		\begin{align}
			\mathcal E_a'(t)
			=\int_MG_a(u)u_t\,\dd\mu 
			=\int_MG_a(u)\Delta w\,\dd\mu 
			=-\int_M
			\langle\nabla G_a(u),\nabla w\rangle\,\dd\mu.
			\label{eq:frequency-Eprime-first}
		\end{align}
		Because
		\[
		\nabla G_a(u)=F'(u)\nabla u=\nabla w,
		\]
		it follows that
		\begin{equation}\label{eq:frequency-Eprime-proof}
			\mathcal E_a'(t)
			=-\int_M|\nabla w|^2\,\dd\mu
			=-\mathcal D(t).
		\end{equation}
		
		To differentiate $\mathcal D$, observe that
		\begin{equation}\label{eq:w-evolution-frequency}
			w_t=F'(u)u_t=F'(u)\Delta w.
		\end{equation}
		Therefore,
		\begin{align}
			\mathcal D'(t)
			=2\int_M\langle\nabla w,\nabla w_t\rangle\,\dd\mu
			=-2\int_M(\Delta w)w_t\,\dd\mu
			=-2\int_MF'(u)(\Delta w)^2\,\dd\mu
			=-2\mathcal A(t).
			\label{eq:frequency-Dprime-proof}
		\end{align}
		
		For completeness, we justify the integrations by parts on a complete
		noncompact manifold without imposing a curvature condition. Choose a
		Lipschitz function $\eta:[0,\infty)\to[0,1]$ such that $\eta=1$ on
		$[0,1]$, $\eta=0$ on $[2,\infty)$, and $|\eta'|\leq C$, and set
		$\chi_R(x)=\eta(d(x,x_0)/R)$. By completeness, $\chi_R$ has compact
		support and satisfies, almost everywhere,
		\[
		0\leq\chi_R\leq1,
		\qquad
		\chi_R=1\ \text{on }B_R,
		\qquad
		\operatorname{supp}\chi_R\subset B_{2R},
		\qquad
		|\nabla\chi_R|\leq\frac{C}{R}.
		\]
		The calculation below is valid with this Lipschitz cutoff; alternatively,
		one may first approximate it by smooth compactly supported cutoffs with the
		same bounds up to an inessential constant.
		Set
		\[
		\mathcal E_{a,R}
		=\int_M\chi_R\Psi_a(u)\,\dd\mu,
		\qquad
		\mathcal D_R
		=\int_M\chi_R|\nabla w|^2\,\dd\mu.
		\]
		Then
		\begin{align}
			\mathcal E_{a,R}'
			=\int_M\chi_RG_a(u)\Delta w\,\dd\mu 
			=-\int_M\chi_R|\nabla w|^2\,\dd\mu
			-\int_MG_a(u)
			\langle\nabla\chi_R,\nabla w\rangle\,\dd\mu.
			\label{eq:frequency-cutoff-E}
		\end{align}
		Fix $I'=[t_1,t_2]\Subset(0,T)$ and put
		$\mathscr A_R=B_{2R}\setminus B_R$. After integrating
		\eqref{eq:frequency-cutoff-E} over $I'$, its boundary term satisfies
		\begin{align}
			\left|
			\int_{t_1}^{t_2}\int_MG_a(u)
			\langle\nabla\chi_R,\nabla w\rangle\,\dd\mu
			\dd t
			\right|
			\leq
			\frac{C}{R}
			\left(
			\int_{t_1}^{t_2}\int_{\mathscr A_R}
			\frac{G_a(u)^2}{F'(u)}\,\dd\mu
			\dd t
			\right)^{1/2} 
			\left(
			\int_{t_1}^{t_2}\int_{\mathscr A_R}
			F'(u)|\nabla w|^2\,\dd\mu
			\dd t
			\right)^{1/2}.
			\label{eq:frequency-cutoff-E-error}
		\end{align}
		Condition \eqref{eq:frequency-admissibility-L2} implies that this
		expression tends to zero as $R\to\infty$.
		
		Similarly, differentiating $\mathcal D_R$ and integrating by parts
		gives
		\begin{align}
			\mathcal D_R'
			&=2\int_M\chi_R
			\langle\nabla w,\nabla w_t\rangle\,\dd\mu
			=-2\int_M\chi_R(\Delta w)w_t\,\dd\mu
			-2\int_Mw_t
			\langle\nabla\chi_R,\nabla w\rangle\,\dd\mu \notag\\
			&=-2\int_M\chi_RF'(u)(\Delta w)^2\,\dd\mu
			-2\int_MF'(u)\Delta w
			\langle\nabla\chi_R,\nabla w\rangle\,\dd\mu.
			\label{eq:frequency-cutoff-D}
		\end{align}
		After integration over $I'$, the boundary term is bounded by
		\begin{align}
			2\left|
			\int_{t_1}^{t_2}\int_MF'(u)\Delta w
			\langle\nabla\chi_R,\nabla w\rangle\,\dd\mu
			\dd t
			\right| 
			\leq
			\frac{2C}{R}
			\left(
			\int_{t_1}^{t_2}\int_{\mathscr A_R}
			F'(u)(\Delta w)^2\,\dd\mu
			\dd t
			\right)^{1/2}
			\left(
			\int_{t_1}^{t_2}\int_{\mathscr A_R}
			F'(u)|\nabla w|^2\,\dd\mu
			\dd t
			\right)^{1/2},
			\label{eq:frequency-cutoff-D-error}
		\end{align}
		which also tends to zero. We now integrate
		\eqref{eq:frequency-cutoff-E} and
		\eqref{eq:frequency-cutoff-D} from $t_1$ to $t_2$ before passing to
		the limit. Dominated convergence, together with
		\eqref{eq:frequency-admissibility-uniform}--
		\eqref{eq:frequency-admissibility-L2}, gives
		\begin{align*}
			\mathcal E_a(t_2)-\mathcal E_a(t_1)
			&=-\int_{t_1}^{t_2}\mathcal D(t)\,\dd t,\\
			\mathcal D(t_2)-\mathcal D(t_1)
			&=-2\int_{t_1}^{t_2}\mathcal A(t)\,\dd t.
		\end{align*}
		Thus $\mathcal E_a$ and $\mathcal D$ are locally absolutely continuous,
		and \eqref{eq:frequency-evolution} holds almost everywhere. Finally,
		\[
		\mathcal E_a''
		=-\mathcal D'
		=2\mathcal A\geq0,
		\]
		which proves the remaining assertions.
	\end{proof}
	
	For $\gamma>0$, on any time interval on which $\mathcal E_a>0$, define
	\begin{equation}\label{eq:frequency-Ngamma}
		\mathcal N_\gamma(t)
		=\frac{\mathcal D(t)}{\mathcal E_a(t)^\gamma}.
	\end{equation}
	Since $\Psi_a\geq0$, the condition $\mathcal E_a(t)>0$ excludes only the
	reference state $u\equiv a$. Proposition~\ref{prop:frequency-evolution}
	shows that $\mathcal E_a,\mathcal D\in AC_{\mathrm{loc}}(0,T)$. Since
	$\mathcal E_a$ has a positive minimum on each compact subinterval under
	consideration, \(\mathcal N_\gamma=\mathcal D\mathcal E_a^{-\gamma}\)
	also belongs to $AC_{\mathrm{loc}}$. Therefore the quotient rule applies
	for almost every time. Differentiating
	\eqref{eq:frequency-Ngamma} and using
	\eqref{eq:frequency-evolution}, we obtain
	\begin{align}
		\mathcal N_\gamma'
		=\mathcal D'\mathcal E_a^{-\gamma}
		-\gamma\mathcal D\mathcal E_a^{-\gamma-1}
		\mathcal E_a' 
		=-2\mathcal A\mathcal E_a^{-\gamma}
		+\gamma\mathcal D^2\mathcal E_a^{-\gamma-1}.
	\end{align}
	Thus, for almost every $t$,
	\begin{equation}\label{eq:frequency-exact-derivative}
		\mathcal N_\gamma'(t)
		=
		\frac{
			\gamma\mathcal D(t)^2
			-2\mathcal E_a(t)\mathcal A(t)}
		{\mathcal E_a(t)^{\gamma+1}}.
	\end{equation}
	
	The following estimate controls the positive term in
	\eqref{eq:frequency-exact-derivative}.
	
	\begin{lemma}[Weighted Cauchy--Schwarz estimate]
		\label{lem:frequency-CS}
		Let $u$ be frequency-admissible relative to $a$. Suppose that
		\begin{equation}\label{eq:frequency-structure}
			G_a(s)^2
			\leq
			\kappa F'(s)\Psi_a(s)
		\end{equation}
		for every $s$ in the range of $u$, where $\kappa>0$. Then, for almost
		every $t\in(0,T)$,
		\begin{equation}\label{eq:frequency-CS}
			\mathcal D(t)^2
			\leq
			\kappa\mathcal E_a(t)\mathcal A(t).
		\end{equation}
	\end{lemma}
	
	\begin{proof}
		We first justify the integration by parts on a complete noncompact
		manifold. For almost every fixed time $t$, let $\chi_R$ be the cutoff used
		in the proof of Proposition~\ref{prop:frequency-evolution}, and set
		$\mathscr A_R=B_{2R}\setminus B_R$. Since
		$\nabla G_a(u)=\nabla w$,
		\[
		\int_M\chi_R|\nabla w|^2\,\dd\mu
		=-\int_M\chi_RG_a(u)\Delta w\,\dd\mu
		-\int_MG_a(u)\langle\nabla\chi_R,\nabla w\rangle\,\dd\mu.
		\]
		The last term satisfies
		\begin{align*}
			\left|\int_MG_a(u)\langle\nabla\chi_R,\nabla w\rangle\,\dd\mu\right|
			&\leq\frac{C}{R}
			\left(\int_{\mathscr A_R}\frac{G_a(u)^2}{F'(u)}\,\dd\mu\right)^{1/2}
			\left(\int_{\mathscr A_R}F'(u)|\nabla w|^2\,\dd\mu\right)^{1/2}
			\longrightarrow0.
		\end{align*}
		Moreover, the admissibility assumptions imply, for almost every $t$,
		that $G_a(u)\Delta w\in L^1(M)$ by the Cauchy--Schwarz inequality.
		Hence dominated convergence gives
		\begin{equation}\label{eq:D-as-Laplacian}
			\mathcal D
			=\int_M|\nabla w|^2\,\dd\mu
			=-\int_MG_a(u)\Delta w\,\dd\mu
		\end{equation}
		for almost every $t$. On a closed manifold, this identity follows
		directly from integration by parts. Since $F'(u)>0$, we may insert the
		weight $F'(u)$ and apply the
		Cauchy--Schwarz inequality:
		\begin{align}
			\mathcal D^2
			=
			\left[
			\int_M
			\frac{G_a(u)}{\sqrt{F'(u)}}
			\bigl(-\sqrt{F'(u)}\,\Delta w\bigr)
			\,\dd\mu
			\right]^2 
			\leq
			\left(
			\int_M\frac{G_a(u)^2}{F'(u)}\,\dd\mu
			\right)
			\left(
			\int_MF'(u)(\Delta w)^2\,\dd\mu
			\right).
			\label{eq:frequency-weighted-CS-step}
		\end{align}
		The structural condition gives
		\[
		\frac{G_a(u)^2}{F'(u)}
		\leq\kappa\Psi_a(u).
		\]
		Consequently,
		\[
		\int_M\frac{G_a(u)^2}{F'(u)}\,\dd\mu
		\leq
		\kappa\int_M\Psi_a(u)\,\dd\mu
		=\kappa\mathcal E_a.
		\]
		Substituting this into
		\eqref{eq:frequency-weighted-CS-step} yields
		\[
		\mathcal D^2
		\leq\kappa\mathcal E_a\mathcal A,
		\]
		as required.
	\end{proof}
	
	\begin{proof}[\bf Proof of Theorem~\ref{thm:frequency-main}]
		By Lemma~\ref{lem:frequency-CS},
		\[
		\mathcal E_a\mathcal A
		\geq\frac{1}{\kappa}\mathcal D^2.
		\]
		Hence the numerator in
		\eqref{eq:frequency-exact-derivative} satisfies
		\begin{align*}
			\gamma\mathcal D^2-2\mathcal E_a\mathcal A
			\leq
			\gamma\mathcal D^2-\frac{2}{\kappa}\mathcal D^2
			=
			\left(\gamma-\frac{2}{\kappa}\right)\mathcal D^2.
		\end{align*}
		Therefore,
		\begin{equation}\label{eq:frequency-differential-bound}
			\mathcal N_\gamma'
			\leq
			\left(\gamma-\frac{2}{\kappa}\right)
			\frac{\mathcal D^2}
			{\mathcal E_a^{\gamma+1}}.
		\end{equation}
		If $0<\gamma\leq2/\kappa$, the right-hand side is nonpositive. Since
		$\mathcal N_\gamma\in AC_{\mathrm{loc}}$ and
		$\mathcal N_\gamma'\leq0$ almost everywhere, $\mathcal N_\gamma$ is
		nonincreasing.
	\end{proof}
	
	\begin{corollary}[Generalized convexity and decay]
		\label{cor:frequency-convexity-decay}
		Under the assumptions of Theorem~\ref{thm:frequency-main}, fix
		$0<\gamma\leq2/\kappa$ and define
		\begin{equation}\label{eq:Phi-gamma}
			\Phi_\gamma(r)=
			\begin{cases}
				\dfrac{r^{1-\gamma}}{1-\gamma},&\gamma\neq1,\\[6pt]
				\log r,&\gamma=1.
			\end{cases}
		\end{equation}
		Then $t\mapsto\Phi_\gamma(\mathcal E_a(t))$ is convex. In particular, if
		$t_1<t<t_2$ and
		$\lambda=(t_2-t)/(t_2-t_1)$, then
		\begin{equation}\label{eq:frequency-three-time}
			\Phi_\gamma(\mathcal E_a(t))
			\leq
			\lambda\Phi_\gamma(\mathcal E_a(t_1))
			+(1-\lambda)\Phi_\gamma(\mathcal E_a(t_2)).
		\end{equation}
		When $\gamma=1$, this becomes the logarithmic convexity estimate
		\begin{equation}\label{eq:frequency-log-convexity}
			\mathcal E_a(t)
			\leq
			\mathcal E_a(t_1)^\lambda
			\mathcal E_a(t_2)^{1-\lambda}.
		\end{equation}
		For $s<t$, one also has
		\begin{equation}\label{eq:frequency-decay-lower}
			\mathcal E_a(t)\geq
			\begin{cases}
				\left[
				\mathcal E_a(s)^{1-\gamma}
				-(1-\gamma)\mathcal N_\gamma(s)(t-s)
				\right]_+^{1/(1-\gamma)},&0<\gamma<1,\\[5pt]
				\mathcal E_a(s)\exp\!\bigl(-\mathcal N_1(s)(t-s)\bigr),&\gamma=1,\\[5pt]
				\left[
				\mathcal E_a(s)^{1-\gamma}
				+(\gamma-1)\mathcal N_\gamma(s)(t-s)
				\right]^{-1/(\gamma-1)},&\gamma>1.
			\end{cases}
		\end{equation}
	\end{corollary}
	
	\begin{proof}
		The chain rule and \eqref{eq:frequency-evolution} give
		\[
		\frac{\dd}{\dd t}\Phi_\gamma(\mathcal E_a(t))
		=-\mathcal N_\gamma(t)
		\]
		for almost every $t$. Hence
		\[
		\frac{\dd^2}{\dd t^2}\Phi_\gamma(\mathcal E_a(t))
		=-\mathcal N_\gamma'(t)\geq0
		\]
		in the distributional sense, which proves convexity and
		\eqref{eq:frequency-three-time}. Since
		$\mathcal N_\gamma(r)\leq\mathcal N_\gamma(s)$ for $r\geq s$, integrating
		$\mathcal E_a'=-\mathcal N_\gamma\mathcal E_a^\gamma$ gives
		\eqref{eq:frequency-decay-lower}.
	\end{proof}
	
	\begin{remark}\label{rem:linear-frequency-comparison}
		If $F(u)=u$ and $a=0$, then
		$\mathcal E_0=\frac12\int_Mu^2\,\dd\mu$ and
		$\mathcal N_1=2\int_M|\nabla u|^2\,\dd\mu/\int_Mu^2\,\dd\mu$.
		Thus \eqref{eq:frequency-three-time} is the usual logarithmic convexity of
		the $L^2$ energy, and the sign convention is the reverse of the frequency
		used in \cite{LiuXu2022}.
	\end{remark}
	
	\begin{remark}\label{rem:frequency-strict}
		If $0<\gamma<2/\kappa$, then
		\eqref{eq:frequency-differential-bound} gives
		$\mathcal N_\gamma'(t)<0$ for almost every $t$ such that
		$\mathcal D(t)>0$. At the critical exponent $\gamma=2/\kappa$, equality
		at almost every time for which $\mathcal D(t)>0$ and $\mathcal A(t)>0$
		requires equality in both the integrated structural estimate and the
		weighted Cauchy--Schwarz inequality. If $\mathcal D(t_0)=0$, then
		$F(u(\cdot,t_0))$ is spatially constant on each connected component of
		$M$. Since $\mathcal D$ is nonnegative and nonincreasing,
		$\mathcal D(t)=\mathcal N_\gamma(t)=0$ for every $t\geq t_0$ for which
		the solution is defined.
	\end{remark}
	
	The structural condition admits a simple sufficient criterion for convex
	filtration laws.
	
	\begin{lemma}[Convex filtration laws]
		\label{lem:frequency-convex}
		Suppose that $F''\geq0$ on $[a,b]$. Then, for every $s\in[a,b]$,
		\begin{equation}\label{eq:frequency-convex-structure}
			\bigl(F(s)-F(a)\bigr)^2
			\leq
			2F'(s)\Psi_a(s).
		\end{equation}
		Consequently, if the range of $u$ is contained in $[a,b]$, then
		\[
		\mathcal N_1(t)
		=\frac{\mathcal D(t)}{\mathcal E_a(t)}
		\]
		is nonincreasing.
	\end{lemma}
	
	\begin{proof}
		Define
		\[
		J(s)
		=2F'(s)\Psi_a(s)-G_a(s)^2.
		\]
		Since $G_a(a)=\Psi_a(a)=0$, we have $J(a)=0$. Using
		\[
		G_a'(s)=F'(s),
		\qquad
		\Psi_a'(s)=G_a(s),
		\]
		we obtain
		\begin{align*}
			J'(s)
			=2F''(s)\Psi_a(s)
			+2F'(s)\Psi_a'(s)
			-2G_a(s)G_a'(s)
			=2F''(s)\Psi_a(s)\geq0.
		\end{align*}
		Thus $J(s)\geq J(a)=0$, which is precisely
		\eqref{eq:frequency-convex-structure}. The last assertion follows
		from Theorem~\ref{thm:frequency-main} with $\kappa=2$ and
		$\gamma=1$.
	\end{proof}
	
	For power laws, the structural ratio can be computed explicitly.
	
	\begin{corollary}[PME/FDE frequency]
		\label{cor:power-frequency}
		Let $F(s)=s^m$, where $m>0$, and take $a=0$. If $u$ is a positive
		frequency-admissible solution, then
		\begin{equation}\label{eq:power-frequency}
			\mathfrak N_m(t)
			=
			\frac{\displaystyle
				\int_M|\nabla u^m|^2\,\dd\mu}
			{\displaystyle
				\left(\int_Mu^{m+1}\,\dd\mu\right)^{2m/(m+1)}}
		\end{equation}
		is nonincreasing.
	\end{corollary}
	
	\begin{proof}
		For $F(s)=s^m$,
		\[
		G_0(s)=s^m,
		\qquad
		\Psi_0(s)=\int_0^sr^m\,\dd r
		=\frac{s^{m+1}}{m+1}.
		\]
		Furthermore,
		\begin{align*}
			\frac{G_0(s)^2}{F'(s)\Psi_0(s)}
			=
			\frac{s^{2m}}
			{m s^{m-1}\,s^{m+1}/(m+1)}
			=\frac{m+1}{m}.
		\end{align*}
		Thus \eqref{eq:frequency-structure} holds with
		\[
		\kappa=\frac{m+1}{m}.
		\]
		The critical exponent is
		\[
		\gamma=\frac{2}{\kappa}
		=\frac{2m}{m+1}.
		\]
		By Theorem~\ref{thm:frequency-main},
		\[
		\frac{\displaystyle\int_M|\nabla u^m|^2\,\dd\mu}
		{\displaystyle
			\left(\frac{1}{m+1}
			\int_Mu^{m+1}\,\dd\mu\right)^{2m/(m+1)}}
		\]
		is nonincreasing. Removing the positive constant factor
		$(m+1)^{2m/(m+1)}$ gives \eqref{eq:power-frequency}.
	\end{proof}

		\section{Examples and geometric interpretation}
	\label{sec:special-laws}
	
	The structural assumptions in the differential Harnack estimate are naturally
	expressed through
	\begin{equation}\label{eq:theta-special}
		\theta(u):=\frac{uF''(u)}{F'(u)}.
	\end{equation}
	Once the uniform parabolicity and barrier assumptions are satisfied, it remains
	to verify
	\begin{equation}\label{eq:structural-condition-special}
		-\frac{2}{n\alpha}
		<\delta_-\leq\theta(u)\leq\delta_+<\infty,
		\qquad \alpha>2.
	\end{equation}
	The following examples show that this condition covers both physically derived
	diffusion laws and non-power laws constructed from their energy structure.
	
	For every strictly increasing filtration law, let $U$ be an internal-energy
	density satisfying
	\begin{equation}\label{eq:internal-energy-special}
		U''(u)=\frac{F'(u)}{u}.
	\end{equation}
	Up to an additive constant, the pressure variable is $v=U'(u)$, and
	\begin{equation}\label{eq:wasserstein-formal-special}
		u_t
		=\operatorname{div}\bigl(u\nabla U'(u)\bigr)
		=\Delta F(u).
	\end{equation}
	Thus the equation is formally the $W_2$-Wasserstein gradient flow of
	$\int_M U(u)\,\dd\mu$; see \cite{Otto2001,McCann1997}.
	
	\subsection{Differential Harnack estimates}
	
	For the power law $F(u)=u^m$, one has
	\[
	F'(u)=mu^{m-1},
	\qquad
	\theta(u)=m-1.
	\]
	The lower condition in \eqref{eq:structural-condition-special} is automatic
	when $m>1$. If $0<m<1$, it becomes
	\[
	m>1-\frac{2}{n\alpha}.
	\]
	For a fixed $m\in(0,1)$, an admissible $\alpha>2$ exists precisely when
	\[
	m>1-\frac1n.
	\]
	The two-sided bound for $F'(u)$ additionally requires the solution range to
	remain in a compact subinterval of $(0,\infty)$. Thus the differential
	Harnack theorem treats the uniformly parabolic positive regime and does not
	include the PME free boundary or the singular FDE behavior near $u=0$.
	
	\paragraph{Linear--quadratic diffusion.}
	Let
	\begin{equation}\label{eq:linear-quadratic-law}
		F(u)=au+bu^2,
		\qquad a,b>0.
	\end{equation}
	Then
	\[
	u_t
	=\operatorname{div}\bigl((a+2bu)\nabla u\bigr),
	\]
	and
	\[
	F'(u)=a+2bu,
	\qquad
	\theta(u)=\frac{2bu}{a+2bu}\in(0,1).
	\]
	Consequently, one may take
	\[
	\delta_-=0,
	\qquad
	\delta_+=1.
	\]
	If $u\leq M$, then
	\[
	a\leq F'(u)\leq a+2bM.
	\]
	After a rescaling of the constants, this diffusion law agrees with the
	continuum equation derived from Brownian particles with short-range repulsive
	interactions in \cite[Eq.~(42)]{BrunaChapmanRobinson2017}. The same structure
	also appears in the corresponding dilute hard-sphere approximation.
	
	The internal energy and pressure may be chosen as
	\begin{equation}\label{eq:linear-quadratic-energy}
		U(u)=a\,u(\log u-1)+bu^2,
		\qquad
		v=U'(u)=a\log u+2bu.
	\end{equation}
	The first term is the Boltzmann entropy, whereas the quadratic term represents
	a local repulsive contribution.
	
	\paragraph{A saturating diffusion coefficient.}
	Consider
	\begin{equation}\label{eq:saturating-law}
		F(u)=au+b\bigl[u-\log(1+u)\bigr],
		\qquad a,b>0.
	\end{equation}
	The equation becomes
	\[
	u_t
	=\operatorname{div}\left[
	\left(a+b\frac{u}{1+u}\right)\nabla u
	\right].
	\]
	A direct calculation gives
	\[
	F'(u)=a+\frac{bu}{1+u},
	\qquad
	F''(u)=\frac{b}{(1+u)^2},
	\]
	and
	\begin{equation}\label{eq:saturating-theta}
		\theta(u)
		=\frac{bu}{(1+u)\,[a+(a+b)u]}.
	\end{equation}
	In particular,
	\[
	a\leq F'(u)\leq a+b,
	\qquad
	0\leq\theta(u)\leq1.
	\]
	Thus one may take
	\[
	h_-=a,
	\qquad
	h_+=a+b,
	\qquad
	\delta_-=0,
	\qquad
	\delta_+=1,
	\]
	independently of the range of the solution. This law describes a baseline
	diffusion supplemented by a density-dependent diffusivity that increases and
	eventually saturates.
	
	Its energy structure is
	\begin{equation}\label{eq:saturating-energy}
		\begin{aligned}
			U(u)
			&=a\,u(\log u-1)
			+b\bigl[(1+u)\log(1+u)-u\bigr],\\
			v=U'(u)
			&=a\log u+b\log(1+u).
		\end{aligned}
	\end{equation}
	Indeed, $uU''(u)=F'(u)$, and hence
	\[
	u_t=\operatorname{div}\bigl(u\nabla U'(u)\bigr).
	\]
	
	\paragraph{A strictly concave non-power law.}
	Let
	\begin{equation}\label{eq:concave-log-law}
		F(u)=au+b\log(1+u),
		\qquad a,b>0.
	\end{equation}
	Although $F''<0$, the equation remains uniformly parabolic on the whole
	positive half-line because
	\[
	F'(u)=a+\frac{b}{1+u},
	\qquad
	a\leq F'(u)\leq a+b.
	\]
	Moreover,
	\begin{equation}\label{eq:concave-log-theta}
		\theta(u)
		=-\frac{bu}{(1+u)[a(1+u)+b]}.
	\end{equation}
	Set
	\begin{equation}\label{eq:kappa-concave-log}
		\kappa
		:=
		\frac{\sqrt{a+b}-\sqrt a}
		{\sqrt{a+b}+\sqrt a}
		\in(0,1).
	\end{equation}
	Differentiating the right-hand side of
	\eqref{eq:concave-log-theta} shows that its minimum is attained at
	\[
	u=\sqrt{1+\frac ba},
	\]
	and
	\begin{equation}\label{eq:concave-log-global-bound}
		-\kappa\leq\theta(u)<0.
	\end{equation}
	Therefore, the lower structural condition is satisfied precisely when
	\begin{equation}\label{eq:concave-log-parameter-condition}
		\kappa<\frac{2}{n\alpha}.
	\end{equation}
	Equivalently,
	\begin{equation}\label{eq:concave-log-ratio-condition}
		\frac ba
		<
		\frac{8n\alpha}{(n\alpha-2)^2}.
	\end{equation}
	For example, if $n=3$, $\alpha=3$, and $a=b=1$, then
	\[
	F(u)=u+\log(1+u),
	\qquad
	\kappa=3-2\sqrt2<\frac29.
	\]
	This gives a concrete strictly concave non-power filtration law satisfying all
	the structural assumptions globally.
	
	The associated energy and pressure are
	\begin{equation}\label{eq:concave-log-energy}
		\begin{aligned}
			U(u)
			&=a\,u(\log u-1)
			+b\bigl[u\log u-(1+u)\log(1+u)\bigr],\\
			v=U'(u)
			&=(a+b)\log u-b\log(1+u).
		\end{aligned}
	\end{equation}
	Although $F$ is strictly concave,
	\[
	U''(u)
	=\frac au+\frac{b}{u(1+u)}>0.
	\]
	Thus the equation still possesses a convex internal energy and a formal
	Wasserstein gradient-flow structure. We regard this as a geometrically
	motivated model; no particular microscopic particle derivation is required.
	
	\paragraph{Volume-exclusion diffusion.}
	The local arguments use $F$ only on an open interval containing the solution
	range. This allows the filtration law
	\begin{equation}\label{eq:volume-exclusion-law}
		F(u)=\frac{Du}{1-cu},
		\qquad
		D,c>0,
		\qquad
		0<u<c^{-1}.
	\end{equation}
	In this case,
	\[
	u_t
	=\operatorname{div}\left(
	\frac{D}{(1-cu)^2}\nabla u
	\right),
	\]
	and
	\[
	F'(u)=\frac{D}{(1-cu)^2},
	\qquad
	\theta(u)=\frac{2cu}{1-cu}>0.
	\]
	If $u\leq M<c^{-1}$, then one may take
	\begin{equation}\label{eq:volume-exclusion-bounds}
		h_-=D,
		\qquad
		h_+=\frac{D}{(1-cM)^2},
		\qquad
		\delta_-=0,
		\qquad
		\delta_+=\frac{2cM}{1-cM}.
	\end{equation}
	The diffusion coefficient diverges as the density approaches the maximal
	packing density $c^{-1}$. In one space dimension, this constitutive law is
	the diffusion part of the continuum limit for Brownian hard rods obtained in
	\cite[Eq.~(3)]{GavishNyquistPeletier2023}. In dimensions $n>2$, it should be
	viewed as the corresponding higher-dimensional model law rather than as a
	direct consequence of that one-dimensional derivation.
	
	An internal energy is
	\begin{equation}\label{eq:volume-exclusion-energy}
		U(u)
		=D\bigl[u(\log u-1)-u\log(1-cu)\bigr].
	\end{equation}
	
	\paragraph{Relation with McCann's displacement-convexity condition.}
	There is a direct connection between the lower bound in
	\eqref{eq:structural-condition-special} and displacement convexity. Suppose
	that $F$ extends continuously to $0$, that $F(0)=0$, and that
	\[
	\theta(u)\geq-\frac{2}{n\alpha}
	\qquad\text{for every }u>0.
	\]
	Set
	\[
	q=\frac{2}{n\alpha}<\frac1n.
	\]
	Then
	\[
	\frac{\dd}{\dd u}\bigl(u^qF'(u)\bigr)
	=u^{q-1}F'(u)\bigl(q+\theta(u)\bigr)\geq0.
	\]
	For $0<s<u$, it follows that
	\[
	F'(s)\leq F'(u)\left(\frac us\right)^q.
	\]
	Integrating from $0$ to $u$ gives
	\[
	F(u)
	\leq
	u^qF'(u)\int_0^u s^{-q}\,\dd s
	=\frac{uF'(u)}{1-q}.
	\]
	Consequently,
	\begin{equation}\label{eq:McCann-from-theta}
		uF'(u)
		\geq(1-q)F(u)
		\geq\left(1-\frac1n\right)F(u),
	\end{equation}
	or equivalently,
	\[
	(n-1)F(u)\leq n uF'(u).
	\]
	In Euclidean space, this is McCann's displacement-convexity condition for the
	internal energy determined by $U''(u)=F'(u)/u$; see
	\cite{McCann1997} and \cite[Eq.~(2)]{BolleyCarrillo2014}. This conclusion
	requires the lower bound for $\theta$ on the entire interval from $0$ to $u$.
	A bound imposed only on a compact solution range does not by itself give a
	global displacement-convexity statement.
	
	\subsection{Frequency monotonicity}
	
	Let $r$ be a reference state and define
	\[
	G_r(s)=F(s)-F(r),
	\qquad
	\Psi_r(s)=\int_r^sG_r(\xi)\,\dd\xi.
	\]
	The structural ratio in the frequency theorem is
	\begin{equation}\label{eq:frequency-ratio-special}
		\mathscr R_r(s)
		=
		\frac{\bigl(F(s)-F(r)\bigr)^2}
		{F'(s)\displaystyle\int_r^s
			\bigl(F(\xi)-F(r)\bigr)\,\dd\xi}.
	\end{equation}
	If $\mathscr R_r\leq\kappa$ on the solution range, then
	\[
	\mathcal N_\gamma(t)
	=
	\frac{\displaystyle\int_M|\nabla F(u)|^2\,\dd\mu}
	{\displaystyle
		\left(\int_M\Psi_r(u)\,\dd\mu\right)^\gamma}
	\]
	is nonincreasing for every $0<\gamma\leq2/\kappa$.
	
	For $F(u)=u^m$ and $r=0$,
	\[
	\Psi_0(u)=\frac{u^{m+1}}{m+1},
	\qquad
	\mathscr R_0(u)=\frac{m+1}{m}.
	\]
	Hence
	\begin{equation}\label{eq:power-frequency-special}
		t\longmapsto
		\frac{\displaystyle\int_M|\nabla u^m|^2\,\dd\mu}
		{\displaystyle
			\left(\int_Mu^{m+1}\,\dd\mu\right)^{2m/(m+1)}}
	\end{equation}
	is nonincreasing for every $m>0$, subject only to the frequency-admissibility
	assumptions. In particular, the restriction $m>1-1/n$ belongs to the
	differential Harnack estimate and is not required for the frequency formula.
	
	The laws \eqref{eq:linear-quadratic-law},
	\eqref{eq:saturating-law}, and
	\eqref{eq:volume-exclusion-law} are convex on their respective
	domains. Suppose, in addition, that the reference state $r$ is a
	lower bound for the solution, so that
	\[
	u(x,t)\geq r
	\]
	throughout the space--time domain under consideration. Then, for every
	$s$ in the range of $u$, one has $s\geq r$. Applying
	Lemma~\ref{lem:frequency-convex} with $a=r$ gives
	\[
	\bigl(F(s)-F(r)\bigr)^2
	\leq 2F'(s)\Psi_r(s).
	\]
	Thus one may take $\kappa=2$, and it follows that
	\begin{equation}\label{eq:convex-frequency-special}
		t\longmapsto
		\frac{\displaystyle\int_M|\nabla F(u)|^2\,\dd\mu}
		{\displaystyle\int_M\Psi_r(u)\,\dd\mu}
	\end{equation}
	is nonincreasing.
	
	If the range of $u$ contains values on both sides of $r$, convexity of
	$F$ alone does not imply the preceding inequality. In that case one
	must instead verify directly that
	\[
	\kappa
	:=
	\sup_{s\in\operatorname{Range}(u)}
	\frac{\bigl(F(s)-F(r)\bigr)^2}
	{F'(s)\Psi_r(s)}
	<\infty.
	\]
	Whenever this condition holds, Theorem~\ref{thm:frequency-main} yields
	the monotonicity of $\mathcal N_\gamma$ for every
	$0<\gamma\leq 2/\kappa$.
	
	The concave law \eqref{eq:concave-log-law} is not covered by this convexity
	criterion. Nevertheless, for $r=0$,
	\[
	\mathscr R_0(s)
	=
	\frac{\bigl(as+b\log(1+s)\bigr)^2}
	{\left(a+\dfrac{b}{1+s}\right)
		\displaystyle\int_0^s
		\bigl(a\xi+b\log(1+\xi)\bigr)\,\dd\xi}
	\]
	is continuous on $(0,\infty)$, and
	\[
	\lim_{s\to0}\mathscr R_0(s)
	=
	\lim_{s\to\infty}\mathscr R_0(s)
	=2.
	\]
	It is therefore bounded on $(0,\infty)$. Thus the general frequency theorem
	also applies to this strictly concave non-power law, with
	\[
	\kappa=\sup_{s>0}\mathscr R_0(s)<\infty.
	\]
	This example illustrates that convexity of $F$ is a convenient sufficient
	condition for frequency monotonicity, rather than a necessary one.
	
	\medskip
	\noindent\textbf{Disclosure on AI assistance.}
	AI-assisted tools, principally ChatGPT 5.6 Sol and ChatGPT-6 Astra, were used to
	check selected calculations and to assist with language editing and
	\LaTeX{} typesetting. The mathematical ideas, results, and proofs were
	developed by the authors. All AI-assisted suggestions were independently
	verified by the authors, who take full responsibility for the content of
	this paper.

	Jian-Hua Hao\\
	\emph{E-mail:} haojianhua@sxu.edu.cn\\
	School of Mathematics and Statistics, Shanxi University, Taiyuan 030006, Shanxi, China.

	Yu-Zhao Wang\\
	\emph{E-mail:} wangyuzhao@sxu.edu.cn\\
	School of Mathematics and Statistics, Shanxi University, Taiyuan 030006, Shanxi, China.
	\normalsize

\end{document}